\documentclass[10pt,a4paper]{article}
\usepackage[utf8]{inputenc}
\usepackage[T1]{fontenc}

\usepackage{amsmath,amssymb,amsthm}
\usepackage{diagbox}

\usepackage{tikz}
\usetikzlibrary{arrows,positioning, shapes.callouts}
\usetikzlibrary{calc,decorations.pathmorphing,shapes}
\usetikzlibrary{backgrounds}
\usepackage{caption}
\usepackage{longtable}
\usepackage{multirow}
\usepackage{array}
\usepackage{colonequals}
\usepackage{graphicx}
\usepackage[caption=false]{subfig}
\usepackage{geometry}
\usepackage{setspace}
\usepackage{enumerate}
\usepackage{enumitem}
\usepackage{algpseudocode}
\usepackage{algorithm}
\usepackage{algorithmicx}
\usepackage{float}
\usepackage{aliascnt}
\usepackage{etoolbox}

\usepackage[para, symbol]{footmisc}
\usepackage{arydshln}

\usepackage{fancyhdr}
\usepackage[colorlinks,
  linkcolor=blue,
  citecolor=blue,
  filecolor=blue,
  menucolor=blue,
  urlcolor=blue,
  pdfnewwindow=true,
  pdftitle ={The\ complete\ enumeration\ of\ the\ 4-polytopes\ and\ 3-spheres\ with\ nine\ vertices},
  pdfauthor= {Moritz\ Firsching}, pdfsubject={}]{hyperref}

\DeclareMathOperator{\conv}{conv}

\DeclareMathOperator{\faces}{faces}
\DeclareMathOperator{\relint}{rel int}

\theoremstyle{definition}
\newtheorem{definition}{Definition}
\theoremstyle{plain}

\newaliascnt{theorem}{definition}
\newtheorem{theorem}[theorem]{Theorem}
\aliascntresetthe{theorem}

\newtheorem{conjecture}[definition]{Conjecture}

\newaliascnt{proposition}{definition}
\newtheorem{proposition}[proposition]{Proposition}
\aliascntresetthe{proposition}

\newaliascnt{lemma}{definition}
\newtheorem{lemma}[lemma]{Lemma}
\aliascntresetthe{lemma}

\newtheorem{corollary}[definition]{Corollary}
\theoremstyle{remark}
\newtheorem{question}[definition]{Question}

\makeatletter
\patchcmd{\ALG@step}{\addtocounter{ALG@line}{1}}{\refstepcounter{ALG@line}}{}{}
\newcommand{\ALG@lineautorefname}{Line}
\makeatother

\DeclareFontFamily{OT1}{pzc}{}
\DeclareFontShape{OT1}{pzc}{m}{it}{<-> s * [1.10] pzcmi7t}{}
\DeclareMathAlphabet{\mathpzc}{OT1}{pzc}{m}{it}

\tikzset{
    >=stealth',
    punkt/.style={
           rectangle,
           rounded corners,
           draw=black, very thick,
           text width=6.5em,
           minimum height=2em,
           text centered},
    pil/.style={
           ->,
           thick,
           shorten <=2pt,
           shorten >=2pt,}
}
\begin{document}
 \title{The complete enumeration of\\ 4-polytopes and 3-spheres with nine vertices}
 \author{Moritz Firsching\thanks{This research was supported by the DFG Collaborative Research Center TRR 109, ‘Discretization in Geometry and Dynamics.’}\\
\small Institut für Mathematik \\[-0.8ex]
\small Freie Universität Berlin\\[-0.8ex]
\small Arnimallee 2\\[-0.8ex]
\small 14195 Berlin\\ [-0.8ex]
\small Germany\\[-0.3ex]
\small \href{mailto:firsching@math.fu-berlin.de}{firsching@math.fu-berlin.de}
}

\date{\today}
\maketitle
\begin{abstract}We describe an algorithm to enumerate polytopes. This algorithm is then implemented to give a complete classification of combinatorial spheres of dimension $3$ with $9$ vertices and decide polytopality of those spheres. In particular, we completely enumerate all combinatorial types of $4$-dimensional polytopes with $9$ vertices. It is shown that all of those combinatorial types are \emph{rational}: They can be realized with rational coordinates. We find $316\,014$ combinatorial spheres on $9$~vertices. Of those, $274\,148$ can be realized as the boundary complex of a four-dimensional polytope and the remaining $41\,866$ are non-polytopal. 
 \end{abstract}
% \tableofcontents
% \newpage
% \LARGE
%  \listoftodos
% \normalsize
% \newpage

% \tableofcontents
 \section{Introduction}
 
 \subsection{Results}Having good examples (and counterexamples) is essential in discrete geometry. To this end, a substantial amount of work has been done on the classification of polytopes and combinatorial spheres; see \autoref{sec:previouswork}. The classification of combinatorial $3$-spheres and $4$-polytopes for $7$~vertices was done by Perles, \cite[Section 6.3]{G67} and for $8$ vertices it was completed by Altshuler and Steinberg \cite{AS85}.

As a next step, we present new algorithmic techniques to obtain a complete classification of combinatorial $3$-spheres with $9$ vertices into polytopes and non-polytopes. 
We obtain the following results:
\begin{theorem}\label{thm:combtypes}
 There are precisely $316\,014$ combinatorial types of combinatorial $3$-spheres with $9$~vertices. 
\end{theorem}
\begin{theorem}\label{thm:polytopes}
 There are precisely $274\,148$ combinatorial types of $4$\nobreakdash-poly\-topes with $9$~vertices.
\end{theorem}

Therefore we have $41\,866$ non-polytopal combinatorial types of combinatorial $3$-spheres with $9$ vertices. 
By taking polar duals, we immediately also obtain a complete classification of  $4$\nobreakdash-polytopes and $3$-spheres with nine \emph{facets}.

We provide rational coordinates for all of the combinatorial types of $4$-polytopes with $9$~vertices. We call a polytope \emph{rational} if it is combinatorially equivalent to a polytope with rational coordinates.
\begin{corollary}
 Every $4$-polytope with up to $9$~vertices is rational. Every $4$-polytope with up to $9$~facets is rational. 
\end{corollary}

Perles showed, using Gale diagrams, that all $d$-polytopes with at most $d+3$ vertices are rational; see \cite[Chapter 6]{G67}. There is an example of Perles of a $8$-polytope with $12$ vertices, which is not rational \cite[Theorem 5.5.4, p.~94]{G67}. For $d=4$, there are examples of non-rational polytopes with $34$ \cite[Corollary of Main Theorem]{RZ95} and $33$ vertices \cite[Thm 9.2.1]{RG06}.

\begin{question}
 What is the smallest $n$, such that there is a \emph{non-rational} $4$-polytope with $n$ vertices?
\end{question}

A list of all combinatorial $3$-spheres with $9$ vertices as well as rational polytopal realizations, if possible, or certificates for their non-polytopality is provided as ancillary data to this arxiv preprint as well as at the author's website:\\ \href{https://page.mi.fu-berlin.de/moritz/}{https://page.mi.fu-berlin.de/moritz/}.
In \autoref{sec:tables}, we summarize our results grouped by number of facets (\autoref{table:facets}), by $f$-vector (\autoref{tab:byfv}) and by flag $f$-vector (\autoref{tab:flag}).

\subsection{Methods}

Enumerating \emph{all} combinatorial $3$-spheres and $4$-poly\-topes is a challenging problem even for a relatively small number of vertices, not only because there is a huge number of them.
In fact, already for $3$-dimensional combinatorial spheres, deciding whether a combinatorial sphere is polytopal is equivalent to the Existential Theory of the Reals \cite{RZ95}. However for few vertices in small dimension a considerable amount of work has been done, for a summary see \autoref{sec:previouswork}.

 In order to enumerate all convex $4$-polytopes with $9$ vertices we proceed in three steps:
 \begin{enumerate}
  \item Completely enumerate combinatorial $3$-spheres (\autoref{sec:genspheres})
  \item Prove non-polytopality for some of them (\autoref{sec:nonrealize})
  \item Provide rational realizations for the rest of them (\autoref{sec2})
 \end{enumerate}
For the first step, we start with a set of \emph{simplicial} spheres and repeatedly untriangulate them. This can be done by joining two facets in the face lattice. It then needs to be checked if the resulting face poset corresponds to a combinatorial sphere. 
For the second step, we resort to the theory of oriented matroids and use Graßmann--Plücker relations to obtain non-realizability certificates.

For the last step, instead of starting with a combinatorial sphere and deciding its realizability as a polytope, we present an algorithm to generate as many different combinatorial types as possible. 
We start with a complete set of realizations of combinatorial types of $4$\nobreakdash-polytopes with less than $9$ vertices and inductively add the $9th$ point at strategic locations. These locations are carefully chosen by considering the hyperplane arrangement generated by the bounding hyperplanes of the polytopes with less than $9$ vertices; compare \cite[Thm 5.2.1.]{G67}.

There is no reason, why this method should generate \emph{all} combinatorial types of $4$-polytopes with $9$ vertices, since it depends on the specific realizations of the polytopes with fewer vertices. This is treated in an exercises by Grünbaum \cite[Ex. 5.2.1]{G67}.
In \cite[Section 5.5]{G67} he explains: 
\begin{quote}
% As a consequence, 
[\dots]
if we are presented with a finite set of polytopes it is
possible to find those among them which are of the same combinatorial
type.
It may seem that this fact, together with theorem 5.2.1 which 
determines all the polytopes obtainable as convex hulls of a given polytope
and one additional point, are sufficient to furnish an \emph{enumeration} of
combinatorial types of $d$-polytopes. [\dots]
% By this we mean a procedure which yields an inductive determination of all $c(k, d)$ combinatorial types of $d$-polytopes which have a given number $k$ of vertices. 
However, from the
result of exercise 5.2.1 it follows that it may be necessary to use different
representatives of a given combinatorial type in order to obtain all the
polytopes having one vertex more which are obtainable from polytopes
of the given combinatorial type. Therefore it is not possible to carry out
the inductive determination of all the combinatorial types in the fashion
suggested above.
\end{quote}
We agree that is in general ``not possible'' to determine all combinatorial types inductively. This is obvious in the presence of non-rational combinatorial types, which can never be generated inductively in this fashion. However, for a  small number of vertices and $d=4$ it turns out that this approach is sufficient to determine \emph{all} combinatorial types.

For each of the steps above we need to be able to quickly decide if a given combinatorial sphere of polytope has been generated before. 
For this it is enough to look at the vertex-facet incidences of the corresponding face lattice. 
Therefore, given two of those face lattices it is sufficient to check if the two directed vertex-facet graphs are isomorphic.
In order to check if a combinatorial type is already contained in a set $\mathcal{S}$ of $N$ combinatorial types, we do not run graph isomorphism $N$~times. 
Instead, we precompute a (hashable) canonical form for all of the graphs in the set. 
Then we can simply check if the canonical form of the vertex-facet graph of the given combinatorial type is in the set of normal forms of graphs in $\mathcal{S}$. 
After computing the canonical form, the average case to check if a graph is in $\mathcal{S}$ will take constant time. 

\subsubsection*{Hardware and computing time}
The computations in \autoref{sec:genspheres} were performed in about 10 hours on a single desktop computer with $8$ cores running at $3.6$GHz with 32GB RAM. The computations for \autoref{sec:nonrealize} and \ref{sec2} were performed on the \emph{allegro} cluster at FU-Berlin, which has about 1000 cores running at $2.6$GHz and having about 3.5TB combined RAM. The results from \autoref{sec:nonrealize} were obtained in about 800 CPU-hours and the results from \autoref{sec2} in about 2000 CPU-hours. We used sagemath \cite{sage} to implement the algorithms described below.

 \subsection{Definitions}

 We assume basic familiarity with convex polytopes; see \cite{G67} and \cite{Z07} for comprehensive introductions. For \autoref{sec:nonrealize}, we assume familiarity with the basic notions of the theory of oriented matroids, especially in the guise of chirotopes; here the standard references include \cite{BVSWZ99}, \cite[Sect. 6]{RGZ04} and \cite{BS89}.
 
 Simplicial spheres, that is, simplicial complexes homeomorphic to a sphere, arise as boundaries of simplicial polytopes. The notion of \emph{combinatorial spheres} is used in slightly different ways in the literature, which is why we give a concise definition below. The intention of the definition is to get a set of \emph{combinatorial spheres} that fits into the following diagram:
 \begin{align*}
 \text{simplicial polytopes} &\subset \text{polytopes}\\
 \rotatebox[origin=c]{-90}{$\subset$}\hspace{1.2cm}&\hspace{1cm} \rotatebox[origin=c]{-90}{$\subset$}\\
 \text{simplicial spheres}&\subset \text{combinatorial spheres}
 \end{align*}
 Here vertical inclusions indicate ``taking boundary'' of the polytopes.
\pagebreak

\begin{definition}[{combinatorial sphere, compare \cite[Def. 2.1]{BZ17}}]
 For $d\in\mathbb{N}$, a \emph{strongly regular} \emph{$d$}-cell complex $\mathcal{C}$ is a finite $d$-dimensional $CW$-complex, that is, a collection of $k$-cells  for $k\leq d$, such that the following two properties hold:
 \begin{enumerate}
  \item \emph{regularity}: the attaching maps of all cells are homeomorphisms also on the boundary.
  \item \emph{intersection}: the intersection of two cells in $\mathcal{C}$ is again a cell in $\mathcal{C}$ (possibly empty). 
  \end{enumerate}
 A strongly regular $d$-cell complex is called a \emph{combinatorial sphere} if it is homeomorphic to $S^d$. 
\end{definition}

It follows that for each $k$-cell $F$, there is a $k$-polytope $P(F)$ and a homeomorphism $h_F$ from $F$ to $H$, such that the preimages of the faces of $P(F)$ under $h_F$ are again cells of $\mathcal{C}$; compare \cite[Section 2]{AS84} and \cite[Section 2]{B73}. 
\begin{definition}[{Eulerian and interval connected; compare \cite[Def. 2.1]{BZ17}}]
 A finite graded lattice of rank $d$ is called \begin{itemize}
                                    \item \emph{Eulerian} if all non-trivial closed intervals have the same number of odd and even elements and it is
                                    \item \emph{interval connected} if all open intervals of length at least $3$ are connected. 
                                   \end{itemize}

\end{definition}

The boundary of a $d$-polytope gives rise to  a combinatorial $(d-1)$-sphere.
The intersection poset of the set of cells of a combinatorial sphere is an Eulerian lattice, which is interval connected. 
Two polytopes (or two combinatorial spheres) are called \emph{combinatorially equivalent} if they give rise to isomorphic face lattices. All properties, that only depend on the isomorphism type of face lattice, such as (flag) $f$-vector, are well defined for combinatorial types of polytopes (or combinatorial spheres). 

\begin{proposition}
 For $d=4$, every interval connected Eulerian lattice of rank $d+1$ is the face lattice of a strongly regular $(d-1)$-cell complex. 
\end{proposition}
For a proof we refer the reader to \cite[Prop. 2.2]{BZ17}.

For $d=4$ it is therefore possible to describe a combinatorial sphere purely in combinatorial terms and it is sufficient to characterize the sphere completely by a set of facets, each containing a set of vertices. This is how we will describe combinatorial spheres below.

 \subsection{Previous results}\label{sec:previouswork}
 Classifications of $(d-1)$-spheres and $d$-poly\-topes with $n$ vertices and have been obtained for various dimensions $d$ and number of vertices $n$. 
 Also certain subfamilies of \emph{all} such spheres and polytopes, namely \emph{simplicial} and \emph{neighborly} ones have been classified.

 In dimension $3$, Steinitz' theorem, \cite[Satz 43, p. 77]{S22} reduces the classification $3$-polytopes to the classification of planar $3$-connected graphs. There are results on the asymptotic behavior; see \cite{B88}, \cite{T80},  \cite{RW82} and \cite[\href{https://oeis.org/A000944}{A944}]{sloane}.
 
If the number of vertices $n$ is less or equal to $d+3$, then every combinatorial $(d-1)$-spheres coincides with the number of $d$-dimensional polytopes and explicit formulas are known; see \cite[Thm. 1]{F06} and \cite[Sect. 6.1]{G67}. The techniques used for these results are Gale-diagrams. 
 \begin{table}[!h] \hspace*{-1cm}
 \begin{tabular}{rl|r|r|r|r|r|r|r}
 &\# of vertices         & 5        &6      &7      &8                      &9                      &10  & 11  \\\hline\hline
& \# of $f$-vectors     & 1        &4      &15     &40                     &\textbf{88}            &?  &?\\
&$3$-spheres             & 1        &4      &31     &\cite{AS85}\hfill\,1\,336    &\textbf{316\,014}    & ? &?\\
&$4$-polytopes\footnotemark[1]      & 1        &4      &31     & \cite{AS85}\hfill\,1\,294   &\textbf{274\,148}    & ? &?\\\hdashline%A005841
&non-polytopal           & 0        &  0    & 0      &42                     &    \textbf{41\,866}    & ? & ?\\\hline
& \# of $f$-vectors     & 1        &   2   & 4    & 7                    &    11        &16  &22\\
simplicial& $3$-spheres  & 1        &2      & 5      &\cite{B73}\hfill 39    & \cite{AS76}\hfill\, 1\,296&\cite{L08}\,247\,882&\cite{SL09}166\,564\,303\\
simplicial &$4$-polytopes\footnotemark[2]& 1        &2      & 5      &\cite{GS67}\hfill 37   &\cite{ABS80}\hfill 1\,142&\cite{F17}\hfill\,162\,004 &?\\\hdashline%A274148
simplicial &non-polytopal&   0      &0      &0      &2                      &154                    &85\,878&?\\\hline
neighborly& $3$-spheres &1          &1      &1      &  \cite{GS67}\hfill 4  &\cite{AS74}\hfill50  &\cite{A77}\hfill3\,540&\cite{SL09}\hfill 897\,819\\ 
neighborly& $4$-polytopes\footnotemark[3]&1          &1      &1      &  \cite{GS67}\hfill 3  &\cite{AS73}\hfill23  &\cite{BSt87}\hfill431&\cite{F17}\hfill13\,935\\\hdashline %A133338
neighborly& non-polytopal &0          &0      &0      &     1               &27                 &           3\,109&           883\,884
\end{tabular}\caption{Classification results for $4$-polytopes with $\leq 11$ vertices. Boldface results are new.}\label{tab:results4d}
\end{table} 
In dimension $4$, we summarize known results in \autoref{tab:results4d}. In each case, the paper we cite is the one that \emph{completes} the classification; in some cases the complete classifications was done a series of papers. 

While various classifications of simplicial and neighborly spheres and polytopes have been obtained in the meantime, the last classification of \emph{all}  $3$-spheres and $4$-polytopes  was completed  by Altshuler and Steinberg  for $8$ vertices in \cite{AS85}. Among other methods, they consider what $3$-polytopes can appear as facets of a $4$-polytope. In our present classification for $9$ vertices we use completely different \emph{algorithmic} methods.

 Recently, Brinkmann and Ziegler enumerated all combinatorial $3$-spheres with $f$-vector \linebreak $(f_0, f_1, f_2, f_3)$, such that $f_0+f_3\leq22$; see \cite[Table 1]{BZ17}. This includes all combinatorial $3$-spheres on $9$ vertices up to $13$ facets. On the other hand, there has been an earlier attempt by Engel to classify \emph{all} combinatorial $3$ spheres with $9$ vertices; see \cite[Table 6]{E1991}. The results of Brinkmann and Ziegler contradict the results  of Engel for the 
 number of combinatorial $3$-spheres with $f$-vector  $(9, f_1, f_2, k)$ for $k\in\{10,11,12,13\}$. 
 Because of this disagreement, it is desirable to have an independent check of the result. We provide this with our results in \autoref{sec:genspheres}.
 Our classification below partially agrees with the results of Brinkmann and Ziegler (for $k<10$ and $k\leq 12$) and partially with those of Engel (for $k<10$ and $k>13$). We explain this in detail at the end of \autoref{sec:genspheres}. 
%  \newpage
\footnotetext[1]{{\cite[\href{https://oeis.org/A005841}{A5841}]{sloane}.}}%
\footnotetext[2]{{\cite[\href{https://oeis.org/A222318}{A222318}]{sloane}.}}%
\footnotetext[3]{{\cite[\href{https://oeis.org/A133338}{A133338}]{sloane}.}}
 \section{Generating combinatorial spheres}\label{sec:genspheres}
 
 We generate a complete set of combinatorial $d$-spheres with $n$ vertices from a complete set of \emph{simplicial} $d$-spheres with $n$ vertices; for each sphere in this set, we generate all spheres obtained by \emph{untriangulating}. By this we mean constructing a new combinatorial sphere from an old one by removing a ridge, that is, a $(d-2)$-dimensional face. 
 A~combinatorial sphere is determined by its face lattice and the face lattice can be completely recovered from the incidence of the atoms and coatoms, that is, from the vertex-facet graph. We consider a combinatorial sphere $M$ as the set of its facets; each facet being a set of vertices. 
 \begin{definition}Let $M$ be a combinatorial sphere and  $f_1, f_2\in M$ two of its facets that intersect in a ridge $r = f_1\cap f_2$. 
 Then the \emph{untriangulation of $M$ with $f_1$ and $f_2$} is the set $U(M)$, obtained from $M$ by replacing $f_1$ and $f_2$ by their union:
 \[U(M)\colonequals \{f_1\cup f_2\}\cup M\setminus \{f_1,f_2\}.\]
 \end{definition}
%  Geometrically this corresponds to the following construction: given two facets $f_1$ and $f_2$ that intersect, we are replacing them by a new facet 
 The untriangulation $U(M)$ might not correspond to a combinatorial sphere. For example, if the new face $f_1\cup f_2$ completely contains another face of $U(M)$ or if there is a face in $M$ that intersects both $f_1$ and $f_2$ but does not intersect $r$, then $U(M)$ will \emph{not} be a combinatorial sphere. In our procedure after generating $U(M)$ it remains to be checked if the corresponding face poset $P(M)$ is graded of rank $d+1$, Eulerian and strongly connected.

 Since every combinatorial sphere can be triangulated (not necessarily in a unique way) until it is a simplicial sphere, all combinatorial spheres can be obtained by repeatedly untriangulating simplicial spheres. 
 In the process of iteratively untriangulating, we might encounter a combinatorial type of combinatorial spheres multiple times. To detect this, we store a canonical form of the (directed) vertex-facet graph, just as we do in \autoref{sec2}. For the same reason we keep a set of such graphs of combinatorial types of posets that do not correspond to combinatorial spheres. If we get such a type when untriangulating, we don't need to untriangulate any further.

 \begin{algorithm}[h!]
 \caption{Enumerating combinatorial spheres}
    \label{genspheres}
     \hspace*{\algorithmicindent}\begin{itemize}[leftmargin=2cm]
     \item[\textbf{Input:}] A dictionary $\mathpzc{SimpSpheres}$ of all \emph{simplicial} $(d-1)$-spheres with $n$ vertices with key-values pairs $(G, S)$, where $S$ is a set of facets, each facet containing a subset of the vertices $\{1, \dots, n\}$ and  $G$ is a canonical form of the vertex-facet graph of $S$.
     \item[\textbf{Output:}] A dictionary $\mathpzc{CombTypes}$ of all \emph{combinatorial} spheres with $n$ vertices. The dictionary $\mathpzc{CombTypes}$ is of the same form as the dictionary $\mathpzc{SimpSpheres}$ given as input.
    \end{itemize}

    \begin{algorithmic}[1] % The number tells where the line numbering should start
         \Procedure{untriangulate}{$M,\mathpzc{CombTypes},\mathpzc{NonTypes}$} \Comment{\dots recursively}
          \State $G\gets$ canonical form of vertex-facet graph of $M$\label{line:canonical2a}%\Comment{\parbox[t]{.34\linewidth}{depends only on the isomorphism class of the graph}}
          \If{($G$ is key of $\mathpzc{CombTypes}$)  \textbf{or}  ($G$ in $\mathpzc{NonTypes}$)} \label{line:ifkeyknown}\Comment{check if we have seen $G$ before}
              \State \Return $\mathpzc{CombTypes}, \mathpzc{NonTypes}$% \Comment{add key-value pair $(G,P)$ to $\mathcal{P}$}
          \Else
             \State $P\gets$ poset of $M$\label{line:poset}
              \If{$P$ is graded of rank $d+1$, Eulerian, strongly connected}
                \State $\mathpzc{CombTypes}[G] \gets M$\Comment{add key-value pair $(G,M)$ to $\mathpzc{CombTypes}$}
                \For {$f_1, f_2 \in \binom{M}{2}$} \Comment{iterate over all pairs of facets}\label{line:for1}
                   \If{ $f_1\cap f_2$ is a ridge in $P$} \Comment{check if $f_1, f_2$ might share a ridge}\label{line:10}
                       \State $U(M)\gets \{f_1\cup f_2\} \cup M\setminus\{f_1,f_2\}$ \Comment{remove facets $f_1, f_2$ and add their union}\label{line:11}
                       \State $\mathpzc{CombTypes}, \mathpzc{NonTypes} \gets$ \Call{untriangulate}{$U(M), \mathpzc{CombTypes},\mathpzc{NonTypes}$}\label{line:recurse}
                    \EndIf
                \EndFor
              \Else
                \State $\mathpzc{NonTypes}\gets \mathpzc{NonTypes}\cup\{G\}$
              \EndIf
     
          \EndIf
          \State \Return $\mathpzc{CombTypes}, \mathpzc{NonTypes}$ 
         \EndProcedure
         \vspace*{1em}
            \State $\mathpzc{CombTypes}\gets $ empty dictionary \Comment{initialize the output dictionary}
            \State $\mathpzc{NonTypes}\gets$ empty set\Comment{initialize the set of non-types}
            \For{$S \in \mathpzc{SimpSpheres}$}\label{line:for2}
               \State $\mathpzc{CombTypes}, \mathpzc{NonTypes} \gets$ \Call{untriangulate}{$S, \mathpzc{CombTypes},\mathpzc{NonTypes}$}
            \EndFor
         \State \Return $\mathpzc{CombTypes}$
    \end{algorithmic}
 \end{algorithm}
 \pagebreak
 This iteration process can be done in multiple ways, we present \autoref{genspheres}, a recursive formulation, and remark:
%  A few remarks about \autoref{genspheres}:
 \begin{description}
 \item[\autoref{line:canonical2a},] Canonical form of vertex-facet graph of $M$: This can be computed by using \emph{bliss}, \cite{bliss}; compare \autoref{genpoly}.
 \item[\autoref{line:poset},] $P\gets$ poset of $M$: We compute the poset by iteratively calculating the intersection of the facets.   
 \item[\autoref{line:10}/\autoref{line:11}:] additional checks if $U(M)$ can possibly be the set of facets of a combinatorial sphere can be added here. 
 \item[\autoref{line:recurse},] the recursion: We know that the recursion terminates, since $U(M)$ has always strictly less elements (facets) than $M$. 
 \item[\autoref{line:for1}/\autoref{line:for2}:] These loops can be parallelized, when keeping multiple copies of $\mathpzc{CombTypes}$ and $\mathpzc{NonTypes}$ and merging them afterwards. 
 \end{description}
 
We use an implementation of \autoref{genspheres} to generate all combinatorial $3$-spheres with up to $9$ vertices. We start from simplicial $3$-spheres: those have been enumerated up to $10$ vertices, see \cite{L08} and for $n\leq 9$, we use the tables provided by Lutz, \cite{L}. The enumeration of the relevant $1296$ simplicial spheres with $9$ has been completed by Altshuler and Steinberg, \cite{AS76}. Since all combinatorial spheres can be obtained by recursively untriangulating, we obtain
  \begingroup
\def\thedefinition{\ref{thm:combtypes}}
\begin{theorem}
 There are precisely $316\,014$ combinatorial types of combinatorial $3$-spheres with $9$ vertices. 
\end{theorem}
\addtocounter{definition}{-1}
\endgroup  
 The method explained above could be summarized as ``flattening a ridge''. The dual operation would be ``edge-reduction'', that is, shrinking an edge until two vertices coincide. This is described by Engel (also in the case of $3$-dimensional polytopes); see \cite[Section 2]{E1991} and \cite{E82}. In fact, in \cite{E1991} an enumeration of all combinatorial $3$-spheres with $9$ facets is attempted. However our results partially disagree with those of Engel. We translate the results of \cite[Table 6]{E1991} in the dual setting; then for each number of facets $5\leq k\leq 27$ we find a number of combinatorial $3$-spheres with $9$ vertices. We compare these numbers to our \autoref{table:facets} and find that the numbers agree for all $k\in\{6,5,8,9\}$ and $14\leq k \leq 27$, but not for $k\in\{10, 11, 12, 13\}$. In all of those cases, Engel claims to have found more combinatorial types of spheres. He does not provide a list of spheres but only their count; therefore we cannot show if he might have counted some of the spheres twice. This might have been the case, because he uses an ad hoc method to determine whether to combinatorial spheres are isomorphic \cite[Section 3]{E1991}, while we reduce the problem to checking if two graphs are isomorphic.  
 
 For up to $13$ facets an enumeration of $3$-spheres with $9$ vertices is done by Brinkmann and Ziegler \cite{BZ17}. Our results agree with theirs for all $k$ except $k=13$. In their paper yet a different method for generating combinatorial spheres is used. They start by generating all possible vertex-edge graphs of possible spheres and then sorting out those that are non-spheres; see \cite[Algorithm 3.1]{BZ17}. While this approach is valid, there seem to have been some problem with the implementation, leading to the inconsistency with our results. However there is an inconsistency only for $f$-vectors with $k=13$ facets and only for 
 for $2$ out of the $6$ $f$-vectors with $k=f_3=13$. This is the largest number of facets Brinkmann and Ziegler consider. For for all other $f$-vectors our results agree.

 In \autoref{tab:BZE}, we give the $f$-vectors and counts of those cases, where our results differ from \cite{BZ17} or \cite{E1991}. Both papers do not attempt to completely decide which of the combinatorial spheres are in fact boundary of polytopes as we do in \autoref{sec2}.

   \begin{table}[h!]\begin{center}
 \begin{tabular}{c|r|r|r}
  $f$-vector & \cite[Table 1]{BZ17}&\autoref{table:facets}&\cite[Table 5]{E1991}\\
  $(9, *, *, 9)$  &1905&1905&1908\\\hline
  $(9, *, *, 10)$ &5376&5376&5411\\\hline
  $(9, *, *, 11)$ &11825&11825&11974\\\hline
  $(9, *, *, 12)$ &20975&20975&21129\\\hline
  $(9, 28, 32, 13)$& 2136 &2224 & \text{not listed}\\\hdashline 
  $(9, 29, 33, 13)$&27 &45 & \text{not listed} \\\hdashline
 $(9, *, *, 13)$ &20871\footnotemark{} & 20975 &21129\\\hline
 \end{tabular}\caption{Inconsistent results for the number of combinatorial types of $3$-spheres with $9$ vertices for some $f$-vectors. For all other $f$-vectors the numbers agree.}\label{tab:BZE}
 \end{center}
\end{table}

 \section{Proving non-polytopality}\label{sec:nonrealize}
 In order to prove that some of the combinatorial $3$-spheres obtained in \autoref{sec:genspheres} are not polytopal, we analyze the orientation information which can be deduced from the combinatorial spheres. Let's consider a combinatorial $3$-sphere $P$, which is realized by the boundary of a $4$-polytope.  \footnotetext{This number is obtained by summing all $f$-vector of the form $(9, *, *, 13)$: $33+1223+7677+9773+2136+27$}
 Then  every ordered set of $5$ vertices of $P$ can be assigned a sign $\{-1,0,1\}$ depending on whether it spans a simplex of negative, zero or positive (signed) volume. The theory of oriented matroids abstracts these concepts and collects the orientation information in the chirotope map $\chi$,  see \cite[Chapter 3.6]{BVSWZ99} for a detailed introduction.
  The following three rules are satisfied by the boundary of a $4$-polytope. To state them, we only need the incidence data, therefore they can be used on combinatorial spheres. 
 \begin{enumerate}
  \item\label{rule1} If five vertices $a,b,c,d,e$ lie in a common facet, then
  $\chi(a,b,c,d,e)=0$.
  \item\label{rule2} If four vertices $a,b,c,d$ lie in a common facet, then for every pair $e,e'$ outside of that facet, we have 
       \[\chi(a,b,c,d,e) = \chi(a,b,c,d,e').\]
  \item\label{rule3} If the tree vertices $a,b,c$ lie in a common ridge $R$, which has the two adjacent facets $F$ and $F'$, such that $R=F\cap F'$, 
        then for every pair $d\in F$, $d'\in F'$ and all $e$ not in $F$  and not in $F'$, we have
         \[\chi(a,b,c,d,e) = -\chi(a,b,c,d',e).\]
 \end{enumerate}
Given a combinatorial $3$-sphere $S$ (as a set of facets), we first find a \emph{partial chirotope}, which can be constructed using the rules above. 
That is, we find subset of $s_0\subset \binom{\text{Vertices of }S}{5}$ and a map $\chi\colon\,s_0\to  \{-1,0,1\}$. 
This can be done first adding all the signs from rule \ref{rule1}, fixing a non-zero sign for an instance of rule \ref{rule3} and then greedily applying rule \ref{rule2} and \ref{rule3} repeatedly. Apart from the choice of the first sign, there is no other choice for the signs defined by $\chi$. Then we can check the \emph{Graßmann--Plücker relations}, which need to be satisfied. The Graßmann--Plücker relations involve $6$ values of the map $\chi$. Of course, with the partial chirotope on $s_0$ we can only check those Graßmann--Plücker relations, where all of those values are defined. If a Graßmann--Plücker relation is violated we can conclude that the sphere in question is not polytopal.
\begin{lemma}\label{lemma1}
  Out of the $316\,014$ combinatorial $3$-spheres with $9$ vertices, there are $24\,028$ spheres, which give rise to a partial chirotope on $s_0$, which contradicts a Graßmann--Plücker relation and are therefore not polytopal.
\end{lemma}
\begin{proof}
 For each of the $3$-spheres we provide the corresponding $s_0$, the partial chirotope and a violating Graßmann--Plücker relation.
\end{proof}

In a next step we seek to enlarge the set $s_0$, where a partial chirotope can be defined. To this end, we consider  Graßmann--Plücker relations with $5$ elements from $s_0$, where the partial chirotope is already defined, 
and one element from $\binom{\text{Vertices of }S}{5}$, where the partial chirotope is not yet known. In some cases we can determine the sign of the new element, add it to $s_0$ and repeat. Iterating this can lead to a contradiction if the combinatorial sphere is not polytopal. 

\begin{lemma}\label{lemma2}
 Out of the $316\,014$ combinatorial $3$-spheres with $9$ vertices, there are $17\,755$ spheres, for which a contradiction arises when completing the partial chirotope on $s_0$. Therefore those $17\,755$ spheres are not polytopal
\end{lemma}
\begin{proof}
 For each of the $3$-sphere we provide the corresponding $s_0$, the chirotope together with a finite list of deductions, each expanding the definition of the partial chirotope to a new element using the Graßmann--Plücker relations until a contradiction is reached.
\end{proof}

In some cases using the method of \autoref{lemma2} does not lead to a contradiction and after a finite number of steps. 
We then have a partial chirotope $\chi$ on a set $s_1$, which contains $s_0$ and which cannot be enlarged by the steps described above. In some cases $s_1$ might be a complete chirotope. 

\begin{lemma}\label{lemma3}
 Out of the $316\,014$ combinatorial spheres  $3$-spheres with $9$ vertices, there are $83$~spheres, for which the partial chirotope $s_0$ is completed to a chirotope on $s_1$, which admits a biquadratic final polynomial. Therefore those $83$ spheres are not polytopal. 
\end{lemma}
\begin{proof}
 For all the relevant $83$ cases, it turns out that $s_1$ is actually a complete chirotope. In each case, we provide the completed chirotope together with the infeasible linear program associated to the biquadratic final polynomial. 
\end{proof}
 Only $11$ of the $83$ cases  already admit a biquadratic final polynomial for the partial chirotope on the set $s_0$. 
 For the other cases there is no biquadratic final polynomial on the the partial chirotope and we need to complete the chirotope in order to prove non-polytopality. 

Since the sets of non-polytopal spheres in \autoref{lemma1}, \autoref{lemma2} and \autoref{lemma3} are disjoint, we obtain
 \begin{theorem}\label{thm:atleast}
  There are at most \[316\,014 - 24\,028 -17\,755 -83  = 274\,148\] $4$-polytopes with $9$ vertices. 
 \end{theorem}

  \section{Generating combinatorial types of polytopes}\label{sec2}
 We describe an algorithm to generate combinatorial types of polytopes. 
 Let $Q$ be a $d$-polytope with $n$ vertices and $k$ facets, let $\mathcal{H}(Q)$ denote the affine hyperplane arrangement consisting of the $k$ hyperplanes supporting the facets of $Q$. We view a supporting hyperplane $h$ as element in $\mathbb{R}\times\mathbb{R}^d$, associated to the hyperplane containing $x\in\mathbb{R}^d$ if and only if their dot product is zero: $(1,x)\cdot h=0.$ The faces of the hyperplane arrangement are given by:
 \[\faces{(\mathcal{H}(Q))}\colonequals \left\{ F_\alpha\subset \mathbb{R}^d \;\middle|\; \alpha = (\alpha_1, \alpha_2, \dots, \alpha_k)\in \{-1,0,1\}^k \text{ if } F_\alpha\neq\varnothing\right\},\]
 where
 \[F_\alpha = \left\{x\in\mathbb{R}^d\; \middle|\; \begin{aligned}(1,x)\cdot h_{i}&\leq 0 &&\text{if }\alpha_i\in\{-1, 0 \}\\(1,x)\cdot h_{i}&\geq 0 &&\text{if }\alpha_i\in\{0,1\} \end{aligned} \text{ for } 1\leq i\leq k\right\}\]
 By definition, the relative interiors of the faces partition $\mathbb{R}^d$:
 \[\mathbb{R}^d = \bigcup^\circ_{F\in\faces(\mathcal{H}(Q))} \relint(F).\]
 \begin{proposition}\label{propo1}

Let $Q$ be a $d$-polytope and $F\in\faces{(\mathcal{H}_j(Q))}$ a face in the hyperplane arrangement of its supporting hyperplanes. Then for any two points $q_1,q_2\in\relint(F)$ in the relative interior of $F$, the polytopes $Q_i\colonequals \conv(Q\cup\{q_i\})$ for $i=1,2$ are combinatorially equivalent.  
\end{proposition}
The proposition is a reformulation of \cite[Thm 5.2.1]{G67} and a proof can be found there.

Motivated by \autoref{propo1}, we proceed inductively to generate combinatorial types of $d\text{-polytopes}$ with $k$ vertices, starting from a set of polytopes with $k-1$ vertices. Given a polytope $Q$  with $k-1$ vertices, we choose  an interior point $p$ from each face of $\mathcal{H}(Q)$. Then we form the convex hull of $Q\cup\{p\}$ and check if this yields a polytope with $k$ vertices. If this is the case, we check if we have seen the combinatorial type of this polytope before. If not, we add it to our output. In order to check quickly if we have already found a combinatorial type, we calculate a canonical form of the (directed) vertex-facet graph, which depends only on the isomorphism class of the graph and therefore only on the combinatorial type of the polytope: it is possible to recover the entire face lattice from the vertex-facet graph. The canonical form can then be used in a hash table or an dictionary.

\begin{algorithm}[H]
    \caption{Generating polytopes}
    \label{genpoly}
%      \hspace*{\algorithmicindent}\textbf{Input:}  An integer $k$ and a set of polytopes with $k-1$ vertices $\mathcal{P}$\\
%     \hspace*{\algorithmicindent} \textbf{Output:} A dictionary $\mathcal{D}$ with key-value pairs $(G,P)$, where $P$ is a polytope with $k$ vertices and
     \hspace*{\algorithmicindent}\begin{itemize}[leftmargin=2cm]
     \item[\textbf{Input:}] An integer $k$ and a set of polytopes with $k-1$ vertices $\mathcal{Q}$
     \item[\textbf{Output:}] A dictionary $\mathcal{P}$ with key-value pairs $(G,P)$, where $P$ is a polytope with $k$ vertices and $G$ is a canonical form of the vertex-facet graph of $P$.
    \end{itemize}

    \begin{algorithmic}[1] % The number tells where the line numbering should start
         \Procedure{update}{$\mathcal{P},P$} \Comment{Update $\mathcal{P}$ with the combinatorial type of $P$.}
          \State $G\gets$ canonical form of vertex-facet graph of $P$\Comment{\parbox[t]{.34\linewidth}{depends only on the isomorphism class of the graph}}\label{line:canonical}
          \If{$G$ is not key of $\mathcal{P}$}\label{line:ifkeynot}
              \State $\mathcal{P}[G]\gets P$\Comment{add key-value pair $(G,P)$ to $\mathcal{P}$}
          \EndIf
         \EndProcedure
         \vspace*{1em}
            \State $\mathcal{P}\gets $ empty dictionary \Comment{initialize the output dictionary}
            \For{$Q \in \mathcal{Q}$} \label{line:qinq}
              \For{$F\in\faces{(\mathcal{H}(Q))}$}\Comment{iterate over all faces in hyperplane arrangement}\label{line:hyperplane}
                 \State $p\gets$ interior point of $F$\Comment{different choices are possible, e.g. center of (bounded) $F$}\label{line:interior}
                 \State $P\gets\conv(Q\cup\{p\})$ 
                 \If{number of vertices of $P = k$}
                        \State  \Call{update}{$\mathcal{P}, P$}
                 \EndIf
              \EndFor
            \EndFor
        \State \Return $\mathcal{P}$
    \end{algorithmic}
\end{algorithm}
% \pagebreak
This is all summarized in \autoref{genpoly}; let us explain some details of this algorithm:
\begin{description}
 \item[\autoref{line:canonical},] canonical form of vertex-facet graph of $P$: This can be computed by using \emph{bliss} \cite{bliss}.
 \item[\autoref{line:ifkeynot}:] if $G$ is already a key in $\mathcal{P}$, that is, there is already a polytope $P'$, which is combinatorially isomorphic to $P$ in the dictionary $\mathcal{P}$, then we could still decide to update the dictionary; for example if $P$ has a shorter description than $P'$, i.e. simpler rational coordinates.
 \item[\autoref{line:qinq}] $Q \in \mathcal{Q}$: at this point the algorithm can be parallelized; each case $Q$ can be run separately yielding a dictionary $\mathcal{P}_Q$. Those must then be collected to give the desired dictionary $\mathcal{P}$.  
 \item[\autoref{line:hyperplane},] $F\in\faces{(\mathcal{H}(Q))}$: the hyperplane arrangement can be computed using \emph{sagemath} methods \cite{sage}.
 \item[\autoref{line:interior},] $p\gets$ interior point of $F$: Here we have some choice for interior point of the rational, not necessarily bounded polyhedron, which is face of the hyperplane arrangement $F$.  To avoid dealing with unbounded polyhedra, we intersect the entire hyperplane arrangement $\mathcal{H}(Q)$ with a rational cuboid (for example axes aligned) that contains all of the vertices of $\mathcal{H}(Q)$ in its interior. The vertices of $\mathcal{H}(Q)$ are its zero-dimensional faces. In general, it has more vertices than $Q$. 
 This reduces the problem to finding a rational point in the relative interior of a rational polytope. We can simply take the barycenter of the vertices of the polyhedron. 
 %However, it is desirable to find a point which has `simple' rational coordinates, for examples consists of rational numbers with small absolute values for numerator and denominator.
\end{description}

We illustrate the procedure in \autoref{figure1} by looking at what happens to an irregular hexagon. Since the classification of polytopes in dimension $2$ is so simple, it might seem like a wasteful way to generate a heptagon, but in higher dimensions the classifications get more interesting.
\begin{figure}[htp]   
\begin{center}
\subfloat[The hyperplane arrangement induced by the facets\dots]{
%  \begin{subfigure}
  \begin{tikzpicture}%
	[scale=.25000000,
	back/.style={loosely dotted, thin},
	edge/.style={color=black!95!black, line width=0.5mm},
	facet/.style={fill=black!95!black,fill opacity=0.00000},
	mainfacet/.style={fill=black!95!black,fill opacity=0.50000},
    nfacet/.style={fill=white!95!black,fill opacity=0.75000},
	vertex/.style={fill=none, circle, thin, draw=none},
	mainvertex/.style={inner sep=1pt,circle,draw=blue!90!black,fill=blue!75!black,thick,anchor=base},
	center2/.style={inner sep=1pt,circle,draw=red!25!black,fill=red!75!black,thick,anchor=base},
	center1/.style={inner sep=1pt,circle,draw=orange!25!black,fill=orange!75!black,thick,anchor=base},
	center0/.style={inner sep=1pt,circle,draw=yellow!25!black,fill=yellow!75!black,thick,anchor=base},
	ncenter/.style={inner sep=3pt,circle,draw=blue!75!black,fill=blue!75!black,line width=0.7mm,anchor=base, opacity=0.50000}
	]
	%\clip (-13.0000000000000, -13.0000000000000) rectangle (20.0000000000000, 12.0000000000000);
\clip (-15.0000000000000, -15.0000000000000) rectangle (22.0000000000000, 14.0000000000000);

 \input{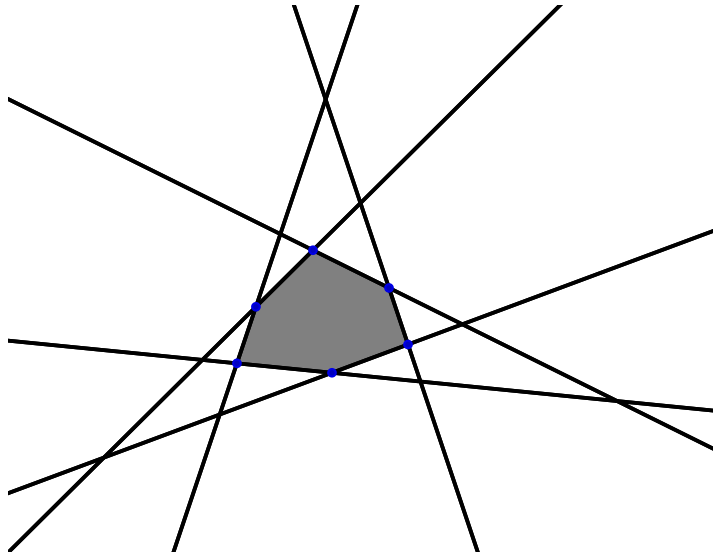}
  \end{tikzpicture} }
  
  \subfloat[\dots in a bounding box with barycenters of all faces.]{
    \begin{tikzpicture}%
	[scale=.25000000,
	back/.style={loosely dotted, thin},
	edge/.style={color=black!95!black, line width=0.5mm},
	facet/.style={fill=black!95!black,fill opacity=0.00000},
	mainfacet/.style={fill=black!95!black,fill opacity=0.50000},
    nfacet/.style={fill=white!95!black,fill opacity=0.75000},
	vertex/.style={fill=none, circle, thin, draw=none},
	mainvertex/.style={inner sep=1pt,circle,draw=blue!90!black,fill=blue!75!black,thick,anchor=base},
	center2/.style={inner sep=1pt,circle,draw=red!25!black,fill=red!75!black,thick,anchor=base},
	center1/.style={inner sep=1pt,circle,draw=orange!25!black,fill=orange!75!black,thick,anchor=base},
	center0/.style={inner sep=1pt,circle,draw=yellow!25!black,fill=yellow!75!black,thick,anchor=base},
	ncenter/.style={inner sep=3pt,circle,draw=blue!75!black,fill=blue!75!black,line width=0.7mm,anchor=base, opacity=0.50000}
	]
 \input{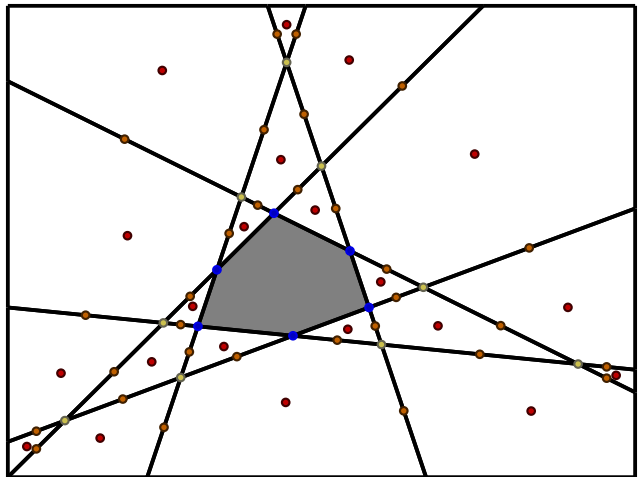}
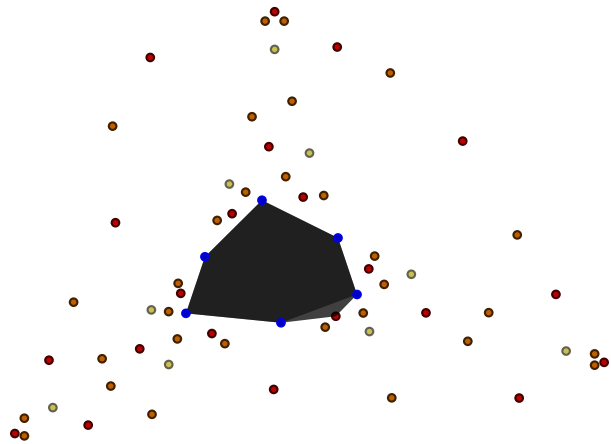
  \end{tikzpicture} }
  
  \subfloat[Adding a point give a new combinatorial type.]{
      \begin{tikzpicture}%
	[scale=.25000000,
	back/.style={loosely dotted, thin},
	edge/.style={color=black!95!black, line width=0.5mm},
	facet/.style={fill=black!95!black,fill opacity=0.00000},
	mainfacet/.style={fill=black!95!black,fill opacity=0.50000},
    nfacet/.style={fill=black!95!black,fill opacity=0.75000},
	vertex/.style={fill=none, circle, thin, draw=none},
	mainvertex/.style={inner sep=1pt,circle,draw=blue!90!black,fill=blue!75!black,thick,anchor=base},
	center2/.style={inner sep=1pt,circle,draw=red!25!black,fill=red!75!black,thick,anchor=base},
	center1/.style={inner sep=1pt,circle,draw=orange!25!black,fill=orange!75!black,thick,anchor=base},
	center0/.style={inner sep=1pt,circle,draw=yellow!25!black,fill=yellow!75!black,thick,anchor=base},
	ncenter/.style={inner sep=3pt,circle,draw=blue!75!black,fill=blue!75!black,line width=0.7mm,anchor=base, opacity=0.50000}
	]
 \useasboundingbox (-13.0000000000000, -13.0000000000000) rectangle (20.0000000000000, 12.0000000000000);
\fill[mainfacet] (-3.00000, -5.00000) -- (-2.00000, -2.00000) -- (1.00000, 1.00000) -- (5.00000, -1.00000) -- (6.00000, -4.00000) -- (2.00000, -5.50000) -- cycle {};

\node[center2] at (-3.272727272727273, -3.9393939393939394){};
\node[center2] at (14.53103448275862, -9.493103448275862){};
\node[center2] at (4.885057471264368, -5.155172413793103){};
\node[center2] at (-1.6349206349206349, -6.071428571428571){};
\node[center2] at (-12.0, -11.375){};
\node[center2] at (-4.876190476190477, 8.571428571428571){};
\node[center2] at (3.1666666666666665, 1.1666666666666667){};
\node[center2] at (16.464285714285715, -3.994642857142857){};
\node[center2] at (-0.5714285714285714, 0.2857142857142857){};
\node[center2] at (9.62807881773399, -4.973522167487685){};
\node[center2] at (6.619047619047619, -2.642857142857143){};
\node[center2] at (1.6167487684729065, -9.035960591133005){};
\node[center2] at (1.3630952380952381, 3.8392857142857144){};
\node[center2] at (-10.204545454545455, -7.485795454545454){};
\node[center2] at (1.6666666666666667, 11.0){};
\node[center2] at (11.55952380952381, 4.136904761904762){};
\node[center2] at (19.0, -7.6){};
\node[center2] at (-8.142857142857142, -10.928571428571429){};
\node[center2] at (4.958333333333333, 9.125){};
\node[center2] at (-5.43073593073593, -6.883116883116883){};
\node[center2] at (-6.706493506493507, -0.19220779220779222){};

\node[center1] at (6.928571428571429, -1.9642857142857142){};
\node[center1] at (-1.3571428571428572, -0.07142857142857142){};
\node[center1] at (7.75, 7.75){};
\node[center1] at (4.327586206896552, -5.732758620689655){};
\node[center1] at (12.928571428571429, -4.964285714285714){};
\node[center1] at (-11.5, -11.5){};
\node[center1] at (7.428571428571429, -3.4642857142857144){};
\node[center1] at (-4.785714285714286, -10.357142857142858){};
\node[center1] at (-11.5, -10.5625){};
\node[center1] at (11.827586206896552, -6.482758620689655){};
\node[center1] at (7.827586206896552, -9.482758620689655){};
\node[center1] at (18.5, -7.15){};
\node[center1] at (-0.9523809523809523, -6.607142857142857){};
\node[center1] at (2.5833333333333335, 6.25){};
\node[center1] at (2.1666666666666665, 10.5){};
\node[center1] at (-3.4523809523809526, -6.357142857142857){};
\node[center1] at (-3.909090909090909, -4.909090909090909){};
\node[center1] at (0.47619047619047616, 5.428571428571429){};
\node[center1] at (2.25, 2.25){};
\node[center1] at (-8.909090909090908, -4.409090909090909){};
\node[center1] at (14.428571428571429, -0.8392857142857143){};
\node[center1] at (-7.409090909090909, -7.409090909090909){};
\node[center1] at (6.327586206896552, -4.982758620689655){};
\node[center1] at (0.14285714285714285, 1.4285714285714286){};
\node[center1] at (-3.409090909090909, -3.409090909090909){};
\node[center1] at (-6.857142857142857, 4.928571428571429){};
\node[center1] at (4.25, 1.25){};
\node[center1] at (1.1666666666666667, 10.5){};
\node[center1] at (-6.9523809523809526, -8.857142857142858){};
\node[center1] at (18.5, -7.75){};

\node[center0] at (8.857142857142858, -2.9285714285714284){};
\node[center0] at (3.5, 3.5){};
\node[center0] at (17.0, -7.0){};
\node[center0] at (-3.9047619047619047, -7.714285714285714){};
\node[center0] at (-10.0, -10.0){};
\node[center0] at (-4.818181818181818, -4.818181818181818){};
\node[center0] at (6.655172413793103, -5.9655172413793105){};
\node[center0] at (1.6666666666666667, 9.0){};
\node[center0] at (-0.7142857142857143, 1.8571428571428572){};

\fill[nfacet] (4.88506, -5.15517) -- (2.00000, -5.50000) -- (-3.00000, -5.00000) -- (-2.00000, -2.00000) -- (1.00000, 1.00000) -- (5.00000, -1.00000) -- (6.00000, -4.00000) -- cycle {};%\node[draw] at (-10,9) {7-gon};

\node[mainvertex] at (-3.00000, -5.00000)     {};
\node[mainvertex] at (-2.00000, -2.00000)     {};
\node[mainvertex] at (6.00000, -4.00000)     {};
\node[mainvertex] at (1.00000, 1.00000)     {};
\node[mainvertex] at (2.00000, -5.50000)     {};
\node[mainvertex] at (5.00000, -1.00000)     {};
  \end{tikzpicture} }
  \end{center}
%  \end{subfigure}
\caption{Generating a heptagon from a hexagon.}\label{figure1} 
\end{figure}

We use an implementation of \autoref{genpoly} for the generation of $4$-polytopes. We start with a realization of the simplex comprised of the origin together with the standard basis vectors of $\mathbb{R}^4$. In running \autoref{genpoly}, there is some choice involved: in \autoref{line:interior}, we choose a point in the interior of a (potentially unbounded) polyhedron.

In a first run, we intersect the unbounded polyhedra with an axes aligned cuboid, which contains all the vertices of $\mathcal{H}(Q)$ with a padding of $1$ unit. For example when going from the simplex in the first step to polytopes with $6$ vertices, we intersection the cells in the hyperplane arrangement with the axes aligned cuboid given by the two coordinates $(-1,-1,-1,-1)$ and $(2,2,2,2)$. Then we choose as an interior point the barycenter. 

In a second run, we choose the same bounding cuboid, but choose the interior point of the bounded polyhedra differently: we strive for comparatively `simple' rational coordinates. For example, we might look for rational numbers with small absolute values for numerator and denominator. (It is of course conceivable to take another definition of `simple'.)
We pick a \emph{subset} of vertices from the set of all vertices of the polyhedron that affinely span the affine hull of the polyhedron. Then we look at the barycenter of this subset of vertices and choose the subset with the `simplest' rational coordinates. 

Putting together the results from these two runs, we obtain
\begin{theorem}
 There are at least $274\,148$ combinatorial types of $4$-polytopes with $9$ vertices.
\end{theorem}
\begin{proof}
 We provide rational coordinates for all combinatorial types in question. 
\end{proof}
This theorem together with \autoref{thm:atleast} implies 
  \begingroup
\def\thedefinition{\ref{thm:polytopes}}
\begin{theorem} There are precisely $274\,148$ combinatorial types of $4$-polytopes with $9$ vertices.
\end{theorem}
\addtocounter{definition}{-1}
\endgroup 
%  \subsection{$4$-polytopes with $9$ vertices}
 \section{Applications}
 The complete classification of combinatorial $3$-spheres and $4$-polytopes with up to $9$ vertices immediately has some applications. We only want to provide two such applications here. 
 
 \subsection{Non-realizable flag \texorpdfstring{$f$}{f}-vectors}
 Recently, Brinkmann and Ziegler provided the first example of a flag $f$-vector of a combinatorial sphere, that does not appear as the flag $f$-vector of a polytope, \cite{BZ15}. Such a flag $f$-vector is called \emph{non-realizable}. The non-realizable flag $f$-vector they provide is $(f_0, f_1, f_2, f_3; f_{02}) = (12, 40, 40, 12; 120)$. There is precisely one sphere, but no polytope with this flag $f$-vector. Our complete classification gives three additional examples of non-realizable flag $f$-vectors. The non-realizable flag $f$-vectors are those in \autoref{tab:flag} that have a ``$0$'' entry in the columns ``$4$-polytopes''. They are $(9, 25, 26, 10; 50)$, $(9, 27, 29, 11; 53)$ and $(9, 27, 30, 12; 57)$. 
 For the last two we have two types of combinatorial spheres and for the first one there is a unique type of combinatorial sphere. We give the sphere as a list of facets in \autoref{tab:nonflag}. (Here vertices are the set $\{1,2,3,4,5,6,7,8,9\}$ and a facet $12345$ is an abbreviation for the set $\{1,2,3,4,5\}$.)
\begin{table}[h!]
 \begin{tabular}{c|l}
  flag $f$-vector &facets of non-realizable $3$-sphere \\\hline
$(9, 25, 26, 10; 50)$ &[12345,12469,12578,12678,13468,1358,23459,25679,346789,35789]\\\hdashline
$(9, 27, 29, 11; 53)$ &[12346,12357,12678,1345,14568,15789,2349,23579,24679,34589,46789]\\
$(9, 27, 29, 11; 53)$ &[12345,12469,12567,13468,13578,1678,23489,2359,25679,35789,46789]\\\hdashline
$(9, 27, 30, 12; 57)$ &[12345,12468,12567,13458,15789,16789,23479,2357,24679,34689,3579,3589]\\
$(9, 27, 30, 12; 57)$ &[1234,12358,1246,12567,13468,15789,16789,23457,24679,34579,34689,3589]\\
\end{tabular} \caption{Non-realizable flag $f$-vectors and spheres with those flag $f$-vectors}\label{tab:nonflag}
\end{table}
%  \subsection{non-inscribable flag \texorpdfstring{$f$}{f}-vectors}
 \subsection{Vertex-edges graphs of polytopes}
 In his PhD-thesis \cite{E14}, Espenschied examines under what circumstances the complete $t$-partite graph $K_{n_1,n_2,\dots, n_t}$ can appear as the vertex-edge graph of a polytope. 
 \begin{conjecture}[{Espenschied's conjecture \cite[p.82]{E14}}]
  If $K_{n_1,n_2,\dots, n_t}$ is the graph of a polytope, then $\{n_1, n_2, \dots n_t\} \subset  \{1, 2\}$ as sets.
 \end{conjecture}
 
 We disprove the above conjecture by looking at the graphs of all $4$-polytopes with $9$ vertices and find a number of counter-examples to this conjecture. In fact, there are $14$ polytopes that contradict Espenschied's conjecture. We list the complete multipartite graphs and the number of combinatorial types of polytopes with that graph in \autoref{tab:espenschied}. See the survey paper by Bayer \cite{Ba17} for more on graphs of polytopes.
 \begin{table} 
\begin{tabular}{l|r|r|r|r}
 Graph&$K_{3, 2, 2, 2}$&$K_{3, 2, 2, 1, 1}$&$K_{3, 2, 1, 1, 1, 1}$&$K_{3, 1, 1, 1, 1, 1, 1}$\\\hline
 number of combinatorial types &1&2&5&6
\end{tabular}
\caption{Number of counter-examples to Espenschied's conjecture}\label{tab:espenschied}
\end{table}

\pagebreak
 \section{Tables of results}\label{sec:tables}
 \begin{table}[!h]
 \begin{tabular}{c|r|r|r|r}
   $f$-vector &\rotatebox{80}{$3$-spheres}& \rotatebox{80}{$4$-polytopes}& \rotatebox{80}{non-realizable}&\rotatebox{80}{\# of $f$-vectors}\\\hline
(5,  *,  *, 5)&  1  &  1  & 0  & 1\\
\hline
(6,  *,  *, 6)&  1  &  1  & 0  & 1\\
(6,  *,  *, 7)&  1  &  1  & 0  & 1\\
(6,  *,  *, 8)&  1  &  1  & 0  & 1\\
(6,  *,  *, 9)&  1  &  1  & 0  & 1\\
\hline
(6,  *,  *, *) & 4 & 4 & 0 & 4\\
\hline\hline
(7,  *,  *, 6)&  1  &  1  & 0  & 1\\
(7,  *,  *, 7)&  3  &  3  & 0  & 2\\
(7,  *,  *, 8)&  5  &  5  & 0  & 2\\
(7,  *,  *, 9)&  7  &  7  & 0  & 2\\
(7,  *,  *, 10)&  6  &  6  & 0  & 2\\
(7,  *,  *, 11)&  4  &  4  & 0  & 2\\
(7,  *,  *, 12)&  3  &  3  & 0  & 2\\
(7,  *,  *, 13)&  1  &  1  & 0  & 1\\
(7,  *,  *, 14)&  1  &  1  & 0  & 1\\
\hline
(7,  *,  *, *) & 31 & 31 & 0 & 15\\
\hline\hline
(8,  *,  *, 6)&  1  &  1  & 0  & 1\\
(8,  *,  *, 7)&  5  &  5  & 0  & 2\\
(8,  *,  *, 8)&  27  &  27  & 0  & 3\\
(8,  *,  *, 9)&  76  &  76  & 0  & 4\\
(8,  *,  *, 10)&  138  &  137  & 1  & 4\\
(8,  *,  *, 11)&  209  &  205  & 4  & 3\\
(8,  *,  *, 12)&  231  &  225  & 6  & 4\\
(8,  *,  *, 13)&  226  &  218  & 8  & 3\\
(8,  *,  *, 14)&  173  &  166  & 7  & 4\\
(8,  *,  *, 15)&  122  &  117  & 5  & 3\\
(8,  *,  *, 16)&  70  &  65  & 5  & 3\\
(8,  *,  *, 17)&  33  &  31  & 2  & 2\\
(8,  *,  *, 18)&  16  &  14  & 2  & 2\\
(8,  *,  *, 19)&  5  &  4  & 1  & 1\\
(8,  *,  *, 20)&  4  &  3  & 1  & 1\\
\hline
(8,  *,  *, *) & 1336 & 1294 & 42 & 40\\
\hline\hline
 \end{tabular}\hfill
 \begin{tabular}{c|r|r|r|r}
   $f$-vector &\rotatebox{80}{$3$-spheres} & \rotatebox{80}{$4$-polytopes}& \rotatebox{80}{non-realizable}&\rotatebox{80}{\# of $f$-vectors}\\\hline
(9,  *,  *, 6)&  1  &  1  & 0  & 1\\
(9,  *,  *, 7)&  7  &  7  & 0  & 2\\
(9,  *,  *, 8)&  76  &  76  & 0  & 4\\
(9,  *,  *, 9)&  467  &  463  & 4  & 6\\
(9,  *,  *, 10)&  1905  &  1872  & 33  & 5\\
(9,  *,  *, 11)&  5376  &  5218  & 158  & 6\\
(9,  *,  *, 12)&  11825  &  11277  & 548  & 6\\
(9,  *,  *, 13)&  20975  &  19666  & 1309  & 6\\
(9,  *,  *, 14)&  31234  &  28821  & 2413  & 5\\
(9,  *,  *, 15)&  39875  &  36105  & 3770  & 6\\
(9,  *,  *, 16)&  44461  &  39436  & 5025  & 5\\
(9,  *,  *, 17)&  43870  &  38007  & 5863  & 6\\
(9,  *,  *, 18)&  38493  &  32492  & 6001  & 5\\
(9,  *,  *, 19)&  30216  &  24741  & 5475  & 5\\
(9,  *,  *, 20)&  21089  &  16747  & 4342  & 4\\
(9,  *,  *, 21)&  13231  &  10069  & 3162  & 4\\
(9,  *,  *, 22)&  7181  &  5306  & 1875  & 3\\
(9,  *,  *, 23)&  3604  &  2468  & 1136  & 3\\
(9,  *,  *, 24)&  1390  &  946  & 444  & 2\\
(9,  *,  *, 25)&  567  &  331  & 236  & 2\\
(9,  *,  *, 26)&  121  &  76  & 45  & 1\\
(9,  *,  *, 27)&  50  &  23  & 27  & 1\\
\hline
$(9,  *,  *, *)$ & 316014 & 274148 & 41866 & 88\\
\hline\hline
 \end{tabular}
 \caption{Combinatorial $3$-spheres and $4$-polytopes with $\leq 9$ vertices, grouped by number of facets.}\label{table:facets}
\end{table}
\begin{table}[!p]\small
\begin{tabular}{c|r|r|r|r}
   $f$-vector &\rotatebox{80}{$3$-spheres} & \rotatebox{80}{$4$-polytopes}& \rotatebox{80}{non-realizable}&\rotatebox{80}{\# of flag $f$-vectors}\\\hline
  (5, 10, 10, 5)&  1 &  1 &  0 &  1\\
\hline
\hline
(6, 13, 13, 6)&  1 &  1 &  0 &  1\\
\hline
(6, 14, 15, 7)&  1 &  1 &  0 &  1\\
\hline
(6, 14, 16, 8)&  1 &  1 &  0 &  1\\
\hline
(6, 15, 18, 9)&  1 &  1 &  0 &  1\\
\hline
\hline
(7, 15, 14, 6)&  1 &  1 &  0 &  1\\
\hline
(7, 16, 16, 7)&  2 &  2 &  0 &  1\\
(7, 17, 17, 7)&  1 &  1 &  0 &  1\\
\hline
(7, 17, 18, 8)&  4 &  4 &  0 &  2\\
(7, 18, 19, 8)&  1 &  1 &  0 &  1\\
\hline
(7, 17, 19, 9)&  1 &  1 &  0 &  1\\
(7, 18, 20, 9)&  6 &  6 &  0 &  2\\
\hline
(7, 18, 21, 10)&  4 &  4 &  0 &  2\\
(7, 19, 22, 10)&  2 &  2 &  0 &  1\\
\hline
(7, 18, 22, 11)&  1 &  1 &  0 &  1\\
(7, 19, 23, 11)&  3 &  3 &  0 &  2\\
\hline
(7, 19, 24, 12)&  2 &  2 &  0 &  1\\
(7, 20, 25, 12)&  1 &  1 &  0 &  1\\
\hline
(7, 20, 26, 13)&  1 &  1 &  0 &  1\\
\hline
(7, 21, 28, 14)&  1 &  1 &  0 &  1\\
\hline
\hline
\end{tabular}\hfill
\begin{tabular}{c|r|r|r|r}
   $f$-vector &\rotatebox{80}{$3$-spheres} & \rotatebox{80}{$4$-polytopes}& \rotatebox{80}{non-realizable}&\rotatebox{80}{\# of flag $f$-vector.}\\\hline
(8, 16, 14, 6)&  1 &  1 &  0 &  1\\
\hline
(8, 18, 17, 7)&  4 &  4 &  0 &  2\\
(8, 19, 18, 7)&  1 &  1 &  0 &  1\\
\hline
(8, 19, 19, 8)&  13 &  13 &  0 &  2\\
(8, 20, 20, 8)&  12 &  12 &  0 &  2\\
(8, 21, 21, 8)&  2 &  2 &  0 &  1\\
\hline
(8, 19, 20, 9)&  1 &  1 &  0 &  1\\
(8, 20, 21, 9)&  31 &  31 &  0 &  2\\
(8, 21, 22, 9)&  37 &  37 &  0 &  3\\
(8, 22, 23, 9)&  7 &  7 &  0 &  2\\
\hline
(8, 20, 22, 10)&  7 &  7 &  0 &  2\\
(8, 21, 23, 10)&  71 &  71 &  0 &  3\\
(8, 22, 24, 10)&  57 &  56 &  1 &  3\\
(8, 23, 25, 10)&  3 &  3 &  0 &  1\\
\hline
(8, 21, 24, 11)&  26 &  26 &  0 &  3\\
(8, 22, 25, 11)&  129 &  128 &  1 &  4\\
(8, 23, 26, 11)&  54 &  51 &  3 &  2\\
\hline
(8, 21, 25, 12)&  4 &  4 &  0 &  1\\
(8, 22, 26, 12)&  75 &  75 &  0 &  4\\
(8, 23, 27, 12)&  133 &  129 &  4 &  3\\
(8, 24, 28, 12)&  19 &  17 &  2 &  1\\
\hline
(8, 22, 27, 13)&  16 &  16 &  0 &  2\\
(8, 23, 28, 13)&  113 &  112 &  1 &  3\\
(8, 24, 29, 13)&  97 &  90 &  7 &  2\\
\hline
(8, 22, 28, 14)&  3 &  3 &  0 &  1\\
(8, 23, 29, 14)&  30 &  30 &  0 &  2\\
(8, 24, 30, 14)&  105 &  103 &  2 &  3\\
(8, 25, 31, 14)&  35 &  30 &  5 &  1\\
\hline
(8, 23, 30, 15)&  5 &  5 &  0 &  1\\
(8, 24, 31, 15)&  39 &  39 &  0 &  2\\
(8, 25, 32, 15)&  78 &  73 &  5 &  2\\
\hline
(8, 24, 32, 16)&  8 &  8 &  0 &  1\\
(8, 25, 33, 16)&  33 &  32 &  1 &  2\\
(8, 26, 34, 16)&  29 &  25 &  4 &  1\\
\hline
(8, 25, 34, 17)&  8 &  8 &  0 &  1\\
(8, 26, 35, 17)&  25 &  23 &  2 &  2\\
\hline
(8, 26, 36, 18)&  6 &  6 &  0 &  1\\
(8, 27, 37, 18)&  10 &  8 &  2 &  1\\
\hline
(8, 27, 38, 19)&  5 &  4 &  1 &  1\\
\hline
(8, 28, 40, 20)&  4 &  3 &  1 &  1\\
\hline
\hline
\end{tabular}  \caption{Combinatorial $3$-spheres and $4$-polytopes with $\leq 8$ vertices, grouped by $f$-vector.}\label{tab:byfv8}
\end{table}
\begin{table}[!p]\small
 \begin{tabular}{c|r|r|r|r}
   $f$-vector &\rotatebox{80}{$3$-spheres} & \rotatebox{80}{$4$-polytopes}& \rotatebox{80}{non-realizable}&\rotatebox{80}{\# of flag $f$-vectors}\\\hline
 (9, 18, 15, 6)&  1 &  1 &  0 &  1\\
\hline
(9, 19, 17, 7)&  1 &  1 &  0 &  1\\
(9, 20, 18, 7)&  6 &  6 &  0 &  2\\
\hline
(9, 20, 19, 8)&  1 &  1 &  0 &  1\\
(9, 21, 20, 8)&  31 &  31 &  0 &  2\\
(9, 22, 21, 8)&  37 &  37 &  0 &  3\\
(9, 23, 22, 8)&  7 &  7 &  0 &  2\\
\hline
(9, 20, 20, 9)&  1 &  1 &  0 &  1\\
(9, 22, 22, 9)&  129 &  129 &  0 &  3\\
(9, 23, 23, 9)&  211 &  209 &  2 &  3\\
(9, 24, 24, 9)&  118 &  116 &  2 &  3\\
(9, 25, 25, 9)&  7 &  7 &  0 &  2\\
(9, 26, 26, 9)&  1 &  1 &  0 &  1\\
\hline
(9, 22, 23, 10)&  12 &  12 &  0 &  3\\
(9, 23, 24, 10)&  398 &  397 &  1 &  4\\
(9, 24, 25, 10)&  904 &  897 &  7 &  4\\
(9, 25, 26, 10)&  524 &  504 &  20 &  4\\
(9, 26, 27, 10)&  67 &  62 &  5 &  3\\
\hline
(9, 23, 25, 11)&  66 &  65 &  1 &  4\\
(9, 24, 26, 11)&  1188 &  1185 &  3 &  4\\
(9, 25, 27, 11)&  2650 &  2593 &  57 &  4\\
(9, 26, 28, 11)&  1344 &  1266 &  78 &  4\\
(9, 27, 29, 11)&  125 &  107 &  18 &  3\\
(9, 28, 30, 11)&  3 &  2 &  1 &  1\\
\hline
(9, 23, 26, 12)&  3 &  3 &  0 &  1\\
(9, 24, 27, 12)&  335 &  333 &  2 &  5\\
(9, 25, 28, 12)&  3275 &  3250 &  25 &  4\\
(9, 26, 29, 12)&  5928 &  5662 &  266 &  5\\
(9, 27, 30, 12)&  2171 &  1943 &  228 &  4\\
(9, 28, 31, 12)&  113 &  86 &  27 &  2\\
\hline
(9, 24, 28, 13)&  33 &  33 &  0 &  2\\
(9, 25, 29, 13)&  1223 &  1219 &  4 &  5\\
(9, 26, 30, 13)&  7677 &  7536 &  141 &  6\\
(9, 27, 31, 13)&  9773 &  9023 &  750 &  5\\
(9, 28, 32, 13)&  2224 &  1829 &  395 &  3\\
(9, 29, 33, 13)&  45 &  26 &  19 &  1\\
\hline
(9, 25, 30, 14)&  205 &  205 &  0 &  3\\
(9, 26, 31, 14)&  3624 &  3608 &  16 &  6\\
(9, 27, 32, 14)&  14312 &  13744 &  568 &  5\\
(9, 28, 33, 14)&  11714 &  10268 &  1446 &  4\\
(9, 29, 34, 14)&  1379 &  996 &  383 &  2\\
\hline
(9, 25, 31, 15)&  15 &  15 &  0 &  1\\
(9, 26, 32, 15)&  771 &  771 &  0 &  4\\
(9, 27, 33, 15)&  7977 &  7878 &  99 &  5\\
(9, 28, 34, 15)&  20764 &  19241 &  1523 &  5\\
(9, 29, 35, 15)&  9961 &  7984 &  1977 &  3\\
(9, 30, 36, 15)&  387 &  216 &  171 &  1\\
\hline
\end{tabular}\hfill
\begin{tabular}{c|r|r|r|r}
   $f$-vector &\rotatebox{80}{$3$-spheres} & \rotatebox{80}{$4$-polytopes}& \rotatebox{80}{non-realizable}&\rotatebox{80}{\# of flag $f$-vectors}\\\hline
(9, 26, 33, 16)&  96 &  96 &  0 &  2\\
(9, 27, 34, 16)&  2038 &  2035 &  3 &  4\\
(9, 28, 35, 16)&  13869 &  13440 &  429 &  5\\
(9, 29, 36, 16)&  22973 &  20057 &  2916 &  4\\
(9, 30, 37, 16)&  5485 &  3808 &  1677 &  2\\
\hline
(9, 26, 34, 17)&  7 &  7 &  0 &  1\\
(9, 27, 35, 17)&  268 &  268 &  0 &  2\\
(9, 28, 36, 17)&  4077 &  4047 &  30 &  4\\
(9, 29, 37, 17)&  19345 &  18090 &  1255 &  4\\
(9, 30, 38, 17)&  18645 &  14763 &  3882 &  3\\
(9, 31, 39, 17)&  1528 &  832 &  696 &  1\\
\hline
(9, 27, 36, 18)&  23 &  23 &  0 &  1\\
(9, 28, 37, 18)&  596 &  596 &  0 &  2\\
(9, 29, 38, 18)&  6671 &  6519 &  152 &  4\\
(9, 30, 39, 18)&  21049 &  18482 &  2567 &  4\\
(9, 31, 40, 18)&  10154 &  6872 &  3282 &  2\\
\hline
(9, 28, 38, 19)&  45 &  45 &  0 &  1\\
(9, 29, 39, 19)&  1061 &  1057 &  4 &  2\\
(9, 30, 40, 19)&  9073 &  8578 &  495 &  4\\
(9, 31, 41, 19)&  17202 &  13559 &  3643 &  3\\
(9, 32, 42, 19)&  2835 &  1502 &  1333 &  1\\
\hline
(9, 29, 40, 20)&  84 &  84 &  0 &  1\\
(9, 30, 41, 20)&  1601 &  1574 &  27 &  2\\
(9, 31, 42, 20)&  9905 &  8793 &  1112 &  3\\
(9, 32, 43, 20)&  9499 &  6296 &  3203 &  2\\
\hline
(9, 30, 42, 21)&  128 &  128 &  0 &  1\\
(9, 31, 43, 21)&  2114 &  2016 &  98 &  2\\
(9, 32, 44, 21)&  8281 &  6536 &  1745 &  3\\
(9, 33, 45, 21)&  2708 &  1389 &  1319 &  1\\
\hline
(9, 31, 44, 22)&  175 &  172 &  3 &  1\\
(9, 32, 45, 22)&  2298 &  2064 &  234 &  2\\
(9, 33, 46, 22)&  4708 &  3070 &  1638 &  2\\
\hline
(9, 32, 46, 23)&  223 &  212 &  11 &  1\\
(9, 33, 47, 23)&  1976 &  1563 &  413 &  2\\
(9, 34, 48, 23)&  1405 &  693 &  712 &  1\\
\hline
(9, 33, 48, 24)&  231 &  209 &  22 &  1\\
(9, 34, 49, 24)&  1159 &  737 &  422 &  2\\
\hline
(9, 34, 50, 25)&  209 &  163 &  46 &  1\\
(9, 35, 51, 25)&  358 &  168 &  190 &  1\\
\hline
(9, 35, 52, 26)&  121 &  76 &  45 &  1\\
\hline
(9, 36, 54, 27)&  50 &  23 &  27 &  1\\
\hline
\hline
\end{tabular} \caption{Combinatorial $3$-spheres and $4$-polytopes with $9$ vertices, grouped by $f$-vector}\label{tab:byfv}
\end{table}

\begin{table}[h!]\footnotesize
\begin{tabular}{c|r|r|r}
  flag $f$-vector &\rotatebox{80}{$3$-spheres} & \rotatebox{80}{$4$-polytopes}& \rotatebox{80}{non-realizable}\\\hline
 $(5,10,10,5;20)$  & $1$ &  $1$ & $0$\\ 
\hline
\hline
$(6,13,13,6;26)$  & $1$ &  $1$ & $0$\\ 
\hline
$(6,14,15,7;29)$  & $1$ &  $1$ & $0$\\ 
\hline
$(6,14,16,8;32)$  & $1$ &  $1$ & $0$\\ 
\hline
$(6,15,18,9;36)$  & $1$ &  $1$ & $0$\\ 
\hline
\hline
$(7,15,14,6;29)$  & $1$ &  $1$ & $0$\\ 
\hline
$(7,16,16,7;32)$  & $2$ &  $2$ & $0$\\ 
\hdashline
$(7,17,17,7;32)$  & $1$ &  $1$ & $0$\\ 
\hline
$(7,17,18,8;35)$  & $3$ &  $3$ & $0$\\ 
$(7,17,18,8;36)$  & $1$ &  $1$ & $0$\\ 
\hdashline
$(7,18,19,8;35)$  & $1$ &  $1$ & $0$\\ 
\hline
$(7,17,19,9;38)$  & $1$ &  $1$ & $0$\\ 
\hdashline
$(7,18,20,9;38)$  & $4$ &  $4$ & $0$\\ 
$(7,18,20,9;39)$  & $2$ &  $2$ & $0$\\ 
\hline
$(7,18,21,10;41)$  & $2$ &  $2$ & $0$\\ 
$(7,18,21,10;42)$  & $2$ &  $2$ & $0$\\ 
\hdashline
$(7,19,22,10;42)$  & $2$ &  $2$ & $0$\\ 
\hline
$(7,18,22,11;44)$  & $1$ &  $1$ & $0$\\ 
\hdashline
$(7,19,23,11;45)$  & $2$ &  $2$ & $0$\\ 
$(7,19,23,11;46)$  & $1$ &  $1$ & $0$\\ 
\hline
$(7,19,24,12;48)$  & $2$ &  $2$ & $0$\\ 
\hdashline
$(7,20,25,12;49)$  & $1$ &  $1$ & $0$\\ 
\hline
$(7,20,26,13;52)$  & $1$ &  $1$ & $0$\\ 
\hline
$(7,21,28,14;56)$  & $1$ &  $1$ & $0$\\ 
\hline
\hline
$(8,16,14,6;32)$  & $1$ &  $1$ & $0$\\ 
\hline
$(8,18,17,7;35)$  & $3$ &  $3$ & $0$\\ 
$(8,18,17,7;36)$  & $1$ &  $1$ & $0$\\ 
\hdashline
$(8,19,18,7;35)$  & $1$ &  $1$ & $0$\\ 
\hline
$(8,19,19,8;38)$  & $12$ &  $12$ & $0$\\ 
$(8,19,19,8;39)$  & $1$ &  $1$ & $0$\\ 
\hdashline
$(8,20,20,8;38)$  & $9$ &  $9$ & $0$\\ 
$(8,20,20,8;39)$  & $3$ &  $3$ & $0$\\ 
\hdashline
$(8,21,21,8;38)$  & $2$ &  $2$ & $0$\\ 
\hline
$(8,19,20,9;41)$  & $1$ &  $1$ & $0$\\ 
\hdashline
$(8,20,21,9;41)$  & $23$ &  $23$ & $0$\\ 
$(8,20,21,9;42)$  & $8$ &  $8$ & $0$\\ 
\hdashline
$(8,21,22,9;41)$  & $20$ &  $20$ & $0$\\ 
$(8,21,22,9;42)$  & $16$ &  $16$ & $0$\\ 
$(8,21,22,9;43)$  & $1$ &  $1$ & $0$\\ 
\hdashline
$(8,22,23,9;41)$  & $5$ &  $5$ & $0$\\ 
$(8,22,23,9;42)$  & $2$ &  $2$ & $0$\\ 
\hline
$(8,20,22,10;44)$  & $6$ &  $6$ & $0$\\ 
$(8,20,22,10;45)$  & $1$ &  $1$ & $0$\\ 
\hdashline
$(8,21,23,10;44)$  & $41$ &  $41$ & $0$\\ 
$(8,21,23,10;45)$  & $23$ &  $23$ & $0$\\ 
$(8,21,23,10;46)$  & $7$ &  $7$ & $0$\\ 
\hdashline
$(8,22,24,10;44)$  & $20$ &  $20$ & $0$\\ 
$(8,22,24,10;45)$  & $35$ &  $35$ & $0$\\ 
$(8,22,24,10;46)$  & $2$ &  $1$ & $1$\\ 
\hdashline
$(8,23,25,10;45)$  & $3$ &  $3$ & $0$\\ 
\hline
\end{tabular}\hfill
\begin{tabular}{c|r|r|r}
   flag $f$-vector &\rotatebox{80}{$3$-spheres} & \rotatebox{80}{$4$-polytopes}& \rotatebox{80}{non-realizable}\\\hline
$(8,21,24,11;47)$  & $17$ &  $17$ & $0$\\ 
$(8,21,24,11;48)$  & $8$ &  $8$ & $0$\\ 
$(8,21,24,11;49)$  & $1$ &  $1$ & $0$\\ 
\hdashline
$(8,22,25,11;47)$  & $38$ &  $38$ & $0$\\ 
$(8,22,25,11;48)$  & $62$ &  $62$ & $0$\\ 
$(8,22,25,11;49)$  & $28$ &  $27$ & $1$\\ 
$(8,22,25,11;50)$  & $1$ &  $1$ & $0$\\ 
\hdashline
$(8,23,26,11;48)$  & $40$ &  $40$ & $0$\\ 
$(8,23,26,11;49)$  & $14$ &  $11$ & $3$\\ 
\hline
$(8,21,25,12;50)$  & $4$ &  $4$ & $0$\\ 
\hdashline
$(8,22,26,12;50)$  & $25$ &  $25$ & $0$\\ 
$(8,22,26,12;51)$  & $32$ &  $32$ & $0$\\ 
$(8,22,26,12;52)$  & $17$ &  $17$ & $0$\\ 
$(8,22,26,12;53)$  & $1$ &  $1$ & $0$\\ 
\hdashline
$(8,23,27,12;51)$  & $58$ &  $58$ & $0$\\ 
$(8,23,27,12;52)$  & $70$ &  $68$ & $2$\\ 
$(8,23,27,12;53)$  & $5$ &  $3$ & $2$\\ 
\hdashline
$(8,24,28,12;52)$  & $19$ &  $17$ & $2$\\ 
\hline
$(8,22,27,13;53)$  & $9$ &  $9$ & $0$\\ 
$(8,22,27,13;54)$  & $7$ &  $7$ & $0$\\ 
\hdashline
$(8,23,28,13;54)$  & $50$ &  $50$ & $0$\\ 
$(8,23,28,13;55)$  & $51$ &  $51$ & $0$\\ 
$(8,23,28,13;56)$  & $12$ &  $11$ & $1$\\ 
\hdashline
$(8,24,29,13;55)$  & $71$ &  $69$ & $2$\\ 
$(8,24,29,13;56)$  & $26$ &  $21$ & $5$\\ 
\hline
$(8,22,28,14;56)$  & $3$ &  $3$ & $0$\\ 
\hdashline
$(8,23,29,14;57)$  & $16$ &  $16$ & $0$\\ 
$(8,23,29,14;58)$  & $14$ &  $14$ & $0$\\ 
\hdashline
$(8,24,30,14;58)$  & $63$ &  $63$ & $0$\\ 
$(8,24,30,14;59)$  & $38$ &  $37$ & $1$\\ 
$(8,24,30,14;60)$  & $4$ &  $3$ & $1$\\ 
\hdashline
$(8,25,31,14;59)$  & $35$ &  $30$ & $5$\\ 
\hline
$(8,23,30,15;60)$  & $5$ &  $5$ & $0$\\ 
\hdashline
$(8,24,31,15;61)$  & $26$ &  $26$ & $0$\\ 
$(8,24,31,15;62)$  & $13$ &  $13$ & $0$\\ 
\hdashline
$(8,25,32,15;62)$  & $61$ &  $59$ & $2$\\ 
$(8,25,32,15;63)$  & $17$ &  $14$ & $3$\\ 
\hline
$(8,24,32,16;64)$  & $8$ &  $8$ & $0$\\ 
\hdashline
$(8,25,33,16;65)$  & $24$ &  $24$ & $0$\\ 
$(8,25,33,16;66)$  & $9$ &  $8$ & $1$\\ 
\hdashline
$(8,26,34,16;66)$  & $29$ &  $25$ & $4$\\ 
\hline
$(8,25,34,17;68)$  & $8$ &  $8$ & $0$\\ 
\hdashline
$(8,26,35,17;69)$  & $20$ &  $19$ & $1$\\ 
$(8,26,35,17;70)$  & $5$ &  $4$ & $1$\\ 
\hline
$(8,26,36,18;72)$  & $6$ &  $6$ & $0$\\ 
\hdashline
$(8,27,37,18;73)$  & $10$ &  $8$ & $2$\\ 
\hline
$(8,27,38,19;76)$  & $5$ &  $4$ & $1$\\ 
\hline
$(8,28,40,20;80)$  & $4$ &  $3$ & $1$\\ 
\hline
\hline
\end{tabular}\caption{Combinatorial $3$-spheres and $4$-polytopes with $\leq 8$ vertices, grouped by flag $f$-vector.}\label{tab:byflfv8}
\end{table}
\begin{table}[!p]\scriptsize
\begin{tabular}{c|r|r|r}
  flag $f$-vector &\rotatebox{80}{$3$-spheres} & \rotatebox{80}{$4$-polytopes}& \rotatebox{80}{non-realizable}\\\hline
$(9,18,15,6;36)$  & $1$ &  $1$ & $0$\\ 
\hline
$(9,19,17,7;38)$  & $1$ &  $1$ & $0$\\ 
\hdashline
$(9,20,18,7;38)$  & $4$ &  $4$ & $0$\\ 
$(9,20,18,7;39)$  & $2$ &  $2$ & $0$\\ 
\hline
$(9,20,19,8;41)$  & $1$ &  $1$ & $0$\\ 
\hdashline
$(9,21,20,8;41)$  & $23$ &  $23$ & $0$\\ 
$(9,21,20,8;42)$  & $8$ &  $8$ & $0$\\ 
\hdashline
$(9,22,21,8;41)$  & $20$ &  $20$ & $0$\\ 
$(9,22,21,8;42)$  & $16$ &  $16$ & $0$\\ 
$(9,22,21,8;43)$  & $1$ &  $1$ & $0$\\ 
\hdashline
$(9,23,22,8;41)$  & $5$ &  $5$ & $0$\\ 
$(9,23,22,8;42)$  & $2$ &  $2$ & $0$\\ 
\hline
$(9,20,20,9;44)$  & $1$ &  $1$ & $0$\\ 
\hdashline
$(9,22,22,9;44)$  & $93$ &  $93$ & $0$\\ 
$(9,22,22,9;45)$  & $32$ &  $32$ & $0$\\ 
$(9,22,22,9;46)$  & $4$ &  $4$ & $0$\\ 
\hdashline
$(9,23,23,9;44)$  & $111$ &  $111$ & $0$\\ 
$(9,23,23,9;45)$  & $90$ &  $90$ & $0$\\ 
$(9,23,23,9;46)$  & $10$ &  $8$ & $2$\\ 
\hdashline
$(9,24,24,9;44)$  & $51$ &  $51$ & $0$\\ 
$(9,24,24,9;45)$  & $63$ &  $63$ & $0$\\ 
$(9,24,24,9;46)$  & $4$ &  $2$ & $2$\\ 
\hdashline
$(9,25,25,9;44)$  & $5$ &  $5$ & $0$\\ 
$(9,25,25,9;45)$  & $2$ &  $2$ & $0$\\ 
\hdashline
$(9,26,26,9;44)$  & $1$ &  $1$ & $0$\\ 
\hline
$(9,22,23,10;47)$  & $8$ &  $8$ & $0$\\ 
$(9,22,23,10;48)$  & $3$ &  $3$ & $0$\\ 
$(9,22,23,10;49)$  & $1$ &  $1$ & $0$\\ 
\hdashline
$(9,23,24,10;47)$  & $242$ &  $242$ & $0$\\ 
$(9,23,24,10;48)$  & $122$ &  $122$ & $0$\\ 
$(9,23,24,10;49)$  & $33$ &  $32$ & $1$\\ 
$(9,23,24,10;50)$  & $1$ &  $1$ & $0$\\ 
\hdashline
$(9,24,25,10;47)$  & $347$ &  $347$ & $0$\\ 
$(9,24,25,10;48)$  & $427$ &  $427$ & $0$\\ 
$(9,24,25,10;49)$  & $128$ &  $121$ & $7$\\ 
$(9,24,25,10;50)$  & $2$ &  $2$ & $0$\\ 
\hdashline
$(9,25,26,10;47)$  & $145$ &  $145$ & $0$\\ 
$(9,25,26,10;48)$  & $311$ &  $311$ & $0$\\ 
$(9,25,26,10;49)$  & $67$ &  $48$ & $19$\\ 
$(9,25,26,10;50)$  & $1$ &  $0$ & $1$\\  
\hdashline
$(9,26,27,10;47)$  & $16$ &  $16$ & $0$\\ 
$(9,26,27,10;48)$  & $42$ &  $42$ & $0$\\ 
$(9,26,27,10;49)$  & $9$ &  $4$ & $5$\\ 
\hline
$(9,23,25,11;50)$  & $51$ &  $51$ & $0$\\ 
$(9,23,25,11;51)$  & $11$ &  $11$ & $0$\\ 
$(9,23,25,11;52)$  & $3$ &  $2$ & $1$\\ 
$(9,23,25,11;53)$  & $1$ &  $1$ & $0$\\ 
\hdashline
$(9,24,26,11;50)$  & $548$ &  $548$ & $0$\\ 
$(9,24,26,11;51)$  & $431$ &  $431$ & $0$\\ 
$(9,24,26,11;52)$  & $196$ &  $194$ & $2$\\ 
$(9,24,26,11;53)$  & $13$ &  $12$ & $1$\\ 
\hdashline
$(9,25,27,11;50)$  & $587$ &  $587$ & $0$\\ 
$(9,25,27,11;51)$  & $1230$ &  $1230$ & $0$\\ 
$(9,25,27,11;52)$  & $777$ &  $741$ & $36$\\ 
$(9,25,27,11;53)$  & $56$ &  $35$ & $21$\\ 
\hdashline
$(9,26,28,11;50)$  & $161$ &  $161$ & $0$\\ 
$(9,26,28,11;51)$  & $715$ &  $715$ & $0$\\ 
$(9,26,28,11;52)$  & $442$ &  $381$ & $61$\\ 
$(9,26,28,11;53)$  & $26$ &  $9$ & $17$\\ 
\hdashline
$(9,27,29,11;51)$  & $70$ &  $70$ & $0$\\ 
$(9,27,29,11;52)$  & $53$ &  $37$ & $16$\\ 
$(9,27,29,11;53)$  & $2$ &  $0$ & $2$\\  
\hdashline
$(9,28,30,11;52)$  & $3$ &  $2$ & $1$\\ 
\hline
\end{tabular}\hfill
\begin{tabular}{c|r|r|r}
   flag $f$-vector &\rotatebox{80}{$3$-spheres} & \rotatebox{80}{$4$-polytopes}& \rotatebox{80}{non-realizable}\\\hline
$(9,23,26,12;53)$  & $3$ &  $3$ & $0$\\ 
\hdashline
$(9,24,27,12;53)$  & $200$ &  $200$ & $0$\\ 
$(9,24,27,12;54)$  & $104$ &  $104$ & $0$\\ 
$(9,24,27,12;55)$  & $25$ &  $24$ & $1$\\ 
$(9,24,27,12;56)$  & $5$ &  $4$ & $1$\\ 
$(9,24,27,12;57)$  & $1$ &  $1$ & $0$\\ 
\hdashline
$(9,25,28,12;53)$  & $834$ &  $834$ & $0$\\ 
$(9,25,28,12;54)$  & $1319$ &  $1319$ & $0$\\ 
$(9,25,28,12;55)$  & $938$ &  $927$ & $11$\\ 
$(9,25,28,12;56)$  & $184$ &  $170$ & $14$\\ 
\hdashline
$(9,26,29,12;53)$  & $487$ &  $487$ & $0$\\ 
$(9,26,29,12;54)$  & $2264$ &  $2264$ & $0$\\ 
$(9,26,29,12;55)$  & $2589$ &  $2496$ & $93$\\ 
$(9,26,29,12;56)$  & $586$ &  $414$ & $172$\\ 
$(9,26,29,12;57)$  & $2$ &  $1$ & $1$\\ 
\hdashline
$(9,27,30,12;54)$  & $692$ &  $692$ & $0$\\ 
$(9,27,30,12;55)$  & $1219$ &  $1121$ & $98$\\ 
$(9,27,30,12;56)$  & $258$ &  $130$ & $128$\\ 
$(9,27,30,12;57)$  & $2$ &  $0$ & $2$\\  
\hdashline
$(9,28,31,12;55)$  & $97$ &  $81$ & $16$\\ 
$(9,28,31,12;56)$  & $16$ &  $5$ & $11$\\ 
\hline
$(9,24,28,13;56)$  & $29$ &  $29$ & $0$\\ 
$(9,24,28,13;57)$  & $4$ &  $4$ & $0$\\ 
\hdashline
$(9,25,29,13;56)$  & $456$ &  $456$ & $0$\\ 
$(9,25,29,13;57)$  & $494$ &  $494$ & $0$\\ 
$(9,25,29,13;58)$  & $232$ &  $231$ & $1$\\ 
$(9,25,29,13;59)$  & $39$ &  $37$ & $2$\\ 
$(9,25,29,13;60)$  & $2$ &  $1$ & $1$\\ 
\hdashline
$(9,26,30,13;56)$  & $683$ &  $683$ & $0$\\ 
$(9,26,30,13;57)$  & $2610$ &  $2610$ & $0$\\ 
$(9,26,30,13;58)$  & $3097$ &  $3063$ & $34$\\ 
$(9,26,30,13;59)$  & $1229$ &  $1134$ & $95$\\ 
$(9,26,30,13;60)$  & $57$ &  $45$ & $12$\\ 
$(9,26,30,13;61)$  & $1$ &  $1$ & $0$\\ 
\hdashline
$(9,27,31,13;57)$  & $1907$ &  $1907$ & $0$\\ 
$(9,27,31,13;58)$  & $4990$ &  $4857$ & $133$\\ 
$(9,27,31,13;59)$  & $2733$ &  $2192$ & $541$\\ 
$(9,27,31,13;60)$  & $141$ &  $66$ & $75$\\ 
$(9,27,31,13;61)$  & $2$ &  $1$ & $1$\\ 
\hdashline
$(9,28,32,13;58)$  & $1252$ &  $1187$ & $65$\\ 
$(9,28,32,13;59)$  & $934$ &  $633$ & $301$\\ 
$(9,28,32,13;60)$  & $38$ &  $9$ & $29$\\ 
\hdashline
$(9,29,33,13;59)$  & $45$ &  $26$ & $19$\\ 
\hline
$(9,25,30,14;59)$  & $122$ &  $122$ & $0$\\ 
$(9,25,30,14;60)$  & $74$ &  $74$ & $0$\\ 
$(9,25,30,14;61)$  & $9$ &  $9$ & $0$\\ 
\hdashline
$(9,26,31,14;59)$  & $466$ &  $466$ & $0$\\ 
$(9,26,31,14;60)$  & $1451$ &  $1451$ & $0$\\ 
$(9,26,31,14;61)$  & $1235$ &  $1232$ & $3$\\ 
$(9,26,31,14;62)$  & $439$ &  $430$ & $9$\\ 
$(9,26,31,14;63)$  & $32$ &  $28$ & $4$\\ 
$(9,26,31,14;64)$  & $1$ &  $1$ & $0$\\ 
\hdashline
$(9,27,32,14;60)$  & $2441$ &  $2441$ & $0$\\ 
$(9,27,32,14;61)$  & $6294$ &  $6236$ & $58$\\ 
$(9,27,32,14;62)$  & $4827$ &  $4500$ & $327$\\ 
$(9,27,32,14;63)$  & $729$ &  $556$ & $173$\\ 
$(9,27,32,14;64)$  & $21$ &  $11$ & $10$\\ 
\hdashline
$(9,28,33,14;61)$  & $4225$ &  $4143$ & $82$\\ 
$(9,28,33,14;62)$  & $6252$ &  $5423$ & $829$\\ 
$(9,28,33,14;63)$  & $1208$ &  $698$ & $510$\\ 
$(9,28,33,14;64)$  & $29$ &  $4$ & $25$\\ 
\hdashline
$(9,29,34,14;62)$  & $1158$ &  $911$ & $247$\\ 
$(9,29,34,14;63)$  & $221$ &  $85$ & $136$\\ 
\hline
\end{tabular}\caption{Combinatorial $3$-spheres and $4$-polytopes with $9$ vertices, grouped by flag $f$-vector.}\label{tab:flag}
\end{table}
\begin{table}[p]\scriptsize
\begin{tabular}{c|r|r|r}
   $f$-vector &\rotatebox{80}{$3$-spheres} & \rotatebox{80}{$4$-polytopes}& \rotatebox{80}{non-realizable}\\\hline
$(9,25,31,15;62)$  & $15$ &  $15$ & $0$\\ 
\hdashline
$(9,26,32,15;62)$  & $188$ &  $188$ & $0$\\ 
$(9,26,32,15;63)$  & $394$ &  $394$ & $0$\\ 
$(9,26,32,15;64)$  & $174$ &  $174$ & $0$\\ 
$(9,26,32,15;65)$  & $15$ &  $15$ & $0$\\ 
\hdashline
$(9,27,33,15;63)$  & $1675$ &  $1675$ & $0$\\ 
$(9,27,33,15;64)$  & $3533$ &  $3525$ & $8$\\ 
$(9,27,33,15;65)$  & $2303$ &  $2254$ & $49$\\ 
$(9,27,33,15;66)$  & $457$ &  $416$ & $41$\\ 
$(9,27,33,15;67)$  & $9$ &  $8$ & $1$\\ 
\hdashline
$(9,28,34,15;64)$  & $5811$ &  $5769$ & $42$\\ 
$(9,28,34,15;65)$  & $10468$ &  $9938$ & $530$\\ 
$(9,28,34,15;66)$  & $4167$ &  $3358$ & $809$\\ 
$(9,28,34,15;67)$  & $315$ &  $175$ & $140$\\ 
$(9,28,34,15;68)$  & $3$ &  $1$ & $2$\\ 
\hdashline
$(9,29,35,15;65)$  & $5763$ &  $5221$ & $542$\\ 
$(9,29,35,15;66)$  & $3911$ &  $2683$ & $1228$\\ 
$(9,29,35,15;67)$  & $287$ &  $80$ & $207$\\ 
\hdashline
$(9,30,36,15;66)$  & $387$ &  $216$ & $171$\\ 
\hline
$(9,26,33,16;65)$  & $42$ &  $42$ & $0$\\ 
$(9,26,33,16;66)$  & $54$ &  $54$ & $0$\\ 
\hdashline
$(9,27,34,16;66)$  & $679$ &  $679$ & $0$\\ 
$(9,27,34,16;67)$  & $961$ &  $961$ & $0$\\ 
$(9,27,34,16;68)$  & $380$ &  $377$ & $3$\\ 
$(9,27,34,16;69)$  & $18$ &  $18$ & $0$\\ 
\hdashline
$(9,28,35,16;67)$  & $4071$ &  $4063$ & $8$\\ 
$(9,28,35,16;68)$  & $6584$ &  $6459$ & $125$\\ 
$(9,28,35,16;69)$  & $2918$ &  $2684$ & $234$\\ 
$(9,28,35,16;70)$  & $292$ &  $232$ & $60$\\ 
$(9,28,35,16;71)$  & $4$ &  $2$ & $2$\\ 
\hdashline
$(9,29,36,16;68)$  & $9767$ &  $9394$ & $373$\\ 
$(9,29,36,16;69)$  & $10885$ &  $9252$ & $1633$\\ 
$(9,29,36,16;70)$  & $2267$ &  $1401$ & $866$\\ 
$(9,29,36,16;71)$  & $54$ &  $10$ & $44$\\ 
\hdashline
$(9,30,37,16;69)$  & $4369$ &  $3332$ & $1037$\\ 
$(9,30,37,16;70)$  & $1116$ &  $476$ & $640$\\ 
\hline
$(9,26,34,17;68)$  & $7$ &  $7$ & $0$\\ 
\hdashline
$(9,27,35,17;69)$  & $153$ &  $153$ & $0$\\ 
$(9,27,35,17;70)$  & $115$ &  $115$ & $0$\\ 
\hdashline
$(9,28,36,17;70)$  & $1635$ &  $1635$ & $0$\\ 
$(9,28,36,17;71)$  & $1886$ &  $1875$ & $11$\\ 
$(9,28,36,17;72)$  & $536$ &  $519$ & $17$\\ 
$(9,28,36,17;73)$  & $20$ &  $18$ & $2$\\ 
\hdashline
$(9,29,37,17;71)$  & $7560$ &  $7445$ & $115$\\ 
$(9,29,37,17;72)$  & $9095$ &  $8500$ & $595$\\ 
$(9,29,37,17;73)$  & $2540$ &  $2054$ & $486$\\ 
$(9,29,37,17;74)$  & $150$ &  $91$ & $59$\\ 
\hdashline
$(9,30,38,17;72)$  & $11070$ &  $9780$ & $1290$\\ 
$(9,30,38,17;73)$  & $7007$ &  $4805$ & $2202$\\ 
$(9,30,38,17;74)$  & $568$ &  $178$ & $390$\\ 
\hdashline
$(9,31,39,17;73)$  & $1528$ &  $832$ & $696$\\ 
\hline
\end{tabular}\hfill
\begin{tabular}{c|r|r|r}
   flag $f$-vector &\rotatebox{80}{$3$-spheres} & \rotatebox{80}{$4$-polytopes}& \rotatebox{80}{non-realizable}\\\hline
$(9,27,36,18;72)$  & $23$ &  $23$ & $0$\\ 
\hdashline
$(9,28,37,18;73)$  & $355$ &  $355$ & $0$\\ 
$(9,28,37,18;74)$  & $241$ &  $241$ & $0$\\ 
\hdashline
$(9,29,38,18;74)$  & $3144$ &  $3130$ & $14$\\ 
$(9,29,38,18;75)$  & $2893$ &  $2815$ & $78$\\ 
$(9,29,38,18;76)$  & $624$ &  $565$ & $59$\\ 
$(9,29,38,18;77)$  & $10$ &  $9$ & $1$\\ 
\hdashline
$(9,30,39,18;75)$  & $10603$ &  $10059$ & $544$\\ 
$(9,30,39,18;76)$  & $8925$ &  $7477$ & $1448$\\ 
$(9,30,39,18;77)$  & $1491$ &  $938$ & $553$\\ 
$(9,30,39,18;78)$  & $30$ &  $8$ & $22$\\ 
\hdashline
$(9,31,40,18;76)$  & $7996$ &  $5923$ & $2073$\\ 
$(9,31,40,18;77)$  & $2158$ &  $949$ & $1209$\\ 
\hline
$(9,28,38,19;76)$  & $45$ &  $45$ & $0$\\ 
\hdashline
$(9,29,39,19;77)$  & $697$ &  $697$ & $0$\\ 
$(9,29,39,19;78)$  & $364$ &  $360$ & $4$\\ 
\hdashline
$(9,30,40,19;78)$  & $4908$ &  $4797$ & $111$\\ 
$(9,30,40,19;79)$  & $3603$ &  $3331$ & $272$\\ 
$(9,30,40,19;80)$  & $557$ &  $447$ & $110$\\ 
$(9,30,40,19;81)$  & $5$ &  $3$ & $2$\\ 
\hdashline
$(9,31,41,19;79)$  & $11005$ &  $9505$ & $1500$\\ 
$(9,31,41,19;80)$  & $5791$ &  $3910$ & $1881$\\ 
$(9,31,41,19;81)$  & $406$ &  $144$ & $262$\\ 
\hdashline
$(9,32,42,19;80)$  & $2835$ &  $1502$ & $1333$\\ 
\hline
$(9,29,40,20;80)$  & $84$ &  $84$ & $0$\\ 
\hdashline
$(9,30,41,20;81)$  & $1111$ &  $1103$ & $8$\\ 
$(9,30,41,20;82)$  & $490$ &  $471$ & $19$\\ 
\hdashline
$(9,31,42,20;82)$  & $6145$ &  $5753$ & $392$\\ 
$(9,31,42,20;83)$  & $3432$ &  $2835$ & $597$\\ 
$(9,31,42,20;84)$  & $328$ &  $205$ & $123$\\ 
\hdashline
$(9,32,43,20;83)$  & $7617$ &  $5458$ & $2159$\\ 
$(9,32,43,20;84)$  & $1882$ &  $838$ & $1044$\\ 
\hline
$(9,30,42,21;84)$  & $128$ &  $128$ & $0$\\ 
\hdashline
$(9,31,43,21;85)$  & $1558$ &  $1509$ & $49$\\ 
$(9,31,43,21;86)$  & $556$ &  $507$ & $49$\\ 
\hdashline
$(9,32,44,21;86)$  & $5986$ &  $5057$ & $929$\\ 
$(9,32,44,21;87)$  & $2189$ &  $1436$ & $753$\\ 
$(9,32,44,21;88)$  & $106$ &  $43$ & $63$\\ 
\hdashline
$(9,33,45,21;87)$  & $2708$ &  $1389$ & $1319$\\ 
\hline
$(9,31,44,22;88)$  & $175$ &  $172$ & $3$\\ 
\hdashline
$(9,32,45,22;89)$  & $1781$ &  $1649$ & $132$\\ 
$(9,32,45,22;90)$  & $517$ &  $415$ & $102$\\ 
\hdashline
$(9,33,46,22;90)$  & $3974$ &  $2742$ & $1232$\\ 
$(9,33,46,22;91)$  & $734$ &  $328$ & $406$\\ 
\hline
$(9,32,46,23;92)$  & $223$ &  $212$ & $11$\\ 
\hdashline
$(9,33,47,23;93)$  & $1657$ &  $1362$ & $295$\\ 
$(9,33,47,23;94)$  & $319$ &  $201$ & $118$\\ 
\hdashline
$(9,34,48,23;94)$  & $1405$ &  $693$ & $712$\\ 
\hline
$(9,33,48,24;96)$  & $231$ &  $209$ & $22$\\ 
\hdashline
$(9,34,49,24;97)$  & $1047$ &  $689$ & $358$\\ 
$(9,34,49,24;98)$  & $112$ &  $48$ & $64$\\ 
\hline
$(9,34,50,25;100)$  & $209$ &  $163$ & $46$\\ 
\hdashline
$(9,35,51,25;101)$  & $358$ &  $168$ & $190$\\ 
\hline
$(9,35,52,26;104)$  & $121$ &  $76$ & $45$\\ 
\hline
$(9,36,54,27;108)$  & $50$ &  $23$ & $27$\\ 
\hline
\hline
\end{tabular}\caption*{Combinatorial $3$-spheres and $4$-polytopes with $9$ vertices, grouped by flag $f$-vector (continued).}
\end{table}

\newpage
\pagebreak

 \section*{Acknowledgments}
 I am very grateful to Günter M.~Ziegler for insightful discussions and suggestions. 
 \bibliography{lit}
 \bibliographystyle{alpha}
\end{document}